\documentclass[a4paper, 11pt]{article}

\usepackage[english]{babel}
\usepackage{amsmath, amssymb, amsthm, amstext, amsfonts, a4}
\usepackage{dsfont}
\usepackage[svgnames]{xcolor}
\usepackage{graphics,graphicx,subfigure}
\usepackage{pgf,tikz}
\usetikzlibrary{backgrounds}
\usetikzlibrary{arrows}

\usepackage{enumerate}
\usepackage{enumitem}

\usepackage{comment}
\usepackage[affil-it]{authblk}
\usepackage{empheq}

\usepackage{hyperref}
\hypersetup{%
colorlinks=true,
linkcolor=Navy,
citecolor=Brown,
urlcolor=Navy
}

\usepackage{cleveref}

\usepackage{geometry}
\numberwithin{equation}{section}

\theoremstyle{plain}
	\newtheorem{theorem}{Theorem}[section]
	
	\newtheorem{lemma}[theorem]{Lemma}

\theoremstyle{definition}
	\newtheorem{definition}[theorem]{Definition}
    
	\newtheorem{remark}{Remark}[section]

\theoremstyle{remark}
    
\newcommand{\fonction}[5]{
    \begin{array}{l|ccc}
        #1: & #2 & \longrightarrow & #3 \\[5pt]
        & \displaystyle{#4} & \longmapsto & \displaystyle{#5}
    \end{array}}

\newcommand{\ds}[1]{\displaystyle{#1}}
\newcommand{\1}{\mathbf{1}}
\renewcommand{\d}[1]{\mathinner{\mathrm{d}{#1}}}
\newcommand{\dd}[1]{\mathinner{\mathrm{d}{#1}}}
\newcommand{\p}{\partial}
\newcommand{\eps}{\mathrm{\varepsilon}}

\newcommand{\e}{\mathrm{e}}
\newcommand{\abs}[1]{\left| #1 \right|}
\newcommand{\norm}[1]{\left\| #1 \right\|}

\newcommand{\N}{\mathbb{N}} 
 
\newcommand{\R}{\mathbb{R}}

\newcommand{\Czero}{\mathbf{C}^{0}}
\newcommand{\Ck}[1]{\mathbf{C}^{#1}}
\newcommand{\Cc}[1]{\mathbf{C}_\mathbf{c}^{#1}}
\newcommand{\Lip}{\mathbf{Lip}}
\newcommand{\LL}[1]{\mathbf{L}^{#1}}
\newcommand{\HH}[1]{\mathbf{H}^{#1}}

\newcommand{\scal}[2]{\left\langle {#1} , {#2} \right\rangle}

\renewcommand{\tilde}[1]{\overset{\sim}{#1}}

\newcommand{\cU}{\mathcal{U}} 
\newcommand{\cF}{\mathcal{F}} 
\newcommand{\cS}{\mathcal{S}} 
\newcommand{\cH}{\mathcal{H}}

\newcommand{\E}{\mathbf{E}}

\newcommand{\bY}{\mathbf{Y}}
\newcommand{\bP}{\mathbf{P}}

\newcommand{\red}[1]{\textbf{\textcolor{Red}{#1}}}

\definecolor{dgreen}{rgb}{0,0.65,0}

\begin{document}

\title{\textbf{Differential Games for a Mixed ODE-PDE System}}

\author{Mauro Garavello$^1$ \qquad Elena Rossi$^2$ \qquad Abraham Sylla$^3$}

\date{ }

\maketitle

\footnotetext[1]{\texttt{mauro.garavello@unimib.it} \\
Department of Mathematics and its Applications,
University of Milano-Bicocca, via R. Cozzi 55,
20125 Milano (Italy)}

\footnotetext[2]{\texttt{elerossi@unimore.it} \\
Department of Sciences and Methods for Engineering,
University of Modena and Reggio Emilia, via Amendola 2, Pad. Morselli
42122 Reggio Emilia (Italy)}

\footnotetext[3]{\texttt{abraham.sylla@u-picardie.fr} \\
LAMFA CNRS UMR 7352, Université de Picardie Jules Verne, 
33 rue Saint-Leu, 80039 Amiens (France)}

\maketitle


\begin{abstract}
  Motivated by a vaccination coverage problem,
  we consider here a zero-sum differential game governed by a differential
  system consisting of a hyperbolic partial differential equation (PDE)
  and an ordinary differential equation (ODE).
  Two players act through their respective controls to influence the evolution
  of the system with the aim of minimizing their
  objective functionals $\mathcal F_1$ and $\mathcal F_2$, under the
  assumption that $\mathcal F_1 + \mathcal F_2 = 0$.

  First we prove a well posedness and a stability result
  for the differential system, once the
  control functions are fixed.
  Then we introduce the concept of non-anticipating strategies for both
  players and we consider the associated value functions, which solve
  two infinite-dimensional Hamilton-Jacobi-Isaacs equations in the viscosity sense.
\end{abstract}







\section{Introduction}
We investigate a two-person zero-sum differential game governed 
by a system consisting of a hyperbolic partial differential equation (PDE)
and an ordinary differential
equation (ODE). In this game, two players use their respective controls
to influence the system's evolution, aiming to optimize their
objective functionals, $\mathcal{F}_1$ and $\mathcal{F}_2$.
Notably, we assume $\mathcal{F}_1 + \mathcal{F}_2 = 0$
in this \emph{zero-sum} game,
implying that one player's gain is the other's loss.

Here we focus on a specific system, mainly motivated by epidemiology,
which aims to consider the 
dispute about the vaccination policies between the scientific community
and groups of people (no-vax) against these practices.
More precisely, 
inspired by the classical SIR model and by the paper~\cite{Fine19861012}
dealing with optimal vaccination policies,
we introduce a simplified model incorporating
vaccination. Let $S = S(t,\xi)$ denote the density at time $t$ of susceptible
individuals, where $\xi \in [0, 1]$ represents the leaning towards vaccination: namely, a value of $\xi = 0$ indicates no vaccination, while $\xi = 1$ signifies guaranteed vaccination. 
Additionally, $I = I(t)$ denotes the
density of infected individuals at time $t$.
The evolution for the $S$ and $I$ populations is described by the following
coupled ODE-PDE system
\begin{equation}
  \label{eq:VaxSystem}
  \left\{
    \begin{array}{rcl@{\qquad}l}
      \ds{\p_t S + \p_\xi \left(g\left(\xi, u_1, u_2\right) S\right)}
      & =
      & - (f(\xi) + \alpha(I)) S,
      & t>0, \, \xi \in (0,1), \\[5pt]
      \dot I
      & =
      & \ds{- \beta I + \alpha (I) \int_0^1 S(t, \xi) \d{\xi}},
      & t>0,
    \end{array}
  \right.
\end{equation}
where $\alpha:[0, \infty) \to [0, \infty)$ is the infection rate function,
$\beta > 0$ is a coefficient taking into
account both the death and the recovery rate of the infected individuals, 
$f = f(\xi)$ is
the vaccination rate for individuals based on their vaccination 
leaning $\xi$, the flux function $g = g(\xi, u_1, u_2)$ describes the change of
individuals' leaning towards vaccinations, and $u_1 \in [0, M_1]$
and $u_2 \in [0, M_2]$
are the control variables,
respectively for the first and for the second player.
Here the first player can represent the government, which, supported by
the scientific community, implements policies to increase the number
of vaccinated people, while the second player is made up of no-vax groups
with the aim of disadvantaging such policies, for example through the use
of social media~\cite{Franceschi2023}.
Game theory for similar topics was already considered in the
literature; see for example~\cite{Chang202057, Kordonis2022, Reluga20101}.
As pointed out in~\cite{Anderson2023},
the influence of anti-vaccine groups is one of several
factors which affects the herd immunity and has a non trivial
economical and healthy cost for the society; see~\cite{Gangarosa1998356}
for the case of the pertussis control.

We consider the Cauchy problem for system~\eqref{eq:VaxSystem} without 
explicitly prescribing any boundary conditions, although the PDE is defined on a bounded domain. 
This is because the chosen flux function $g$ inherently satisfies
zero-flux boundary conditions.
First we prove the well posedness of the Cauchy problem~\eqref{eq:VaxSystem},
both in $\LL1$ and in $\LL2$,
for any given choice of the control functions.
Then we pass to a game perspective and we construct the value
functions for both players, using non anticipating strategies, which, roughly
speaking, are strategies based on the past and present status, but not on the
future one. 
The main result consists in the proof that each value
function satisfies a suitable Hamilton-Jacobi-Bellman (HJB) infinite-dimensional
equation in the viscosity sense. The key tool here is a
\textsl{dynamical programming principle}, satisfied by each value
function. 
The Hamiltonian function is strictly related to
the choice of a duality, which, in the present
infinite-dimensional setting, is not straightforward.
Indeed, the natural space for the unknown $S$ in~\eqref{eq:VaxSystem}
is $\LL1$ in the variable $\xi$, as usual for hyperbolic conservation laws.
However, the Banach space $\LL1$ lacks reflexivity, since
the bi-dual space strictly contains $\LL1$.
Hence we use
here the $\LL2$ duality. 
This has the consequence that we need to prove the
well posedness of the PDE equation in~\eqref{eq:VaxSystem}
in the $\HH1$ setting, since, otherwise, the term
$\partial_\xi \left(g\left(\xi, u_1, u_2\right) S\right)$
is not well defined in $\LL2$.

As intermediate results we also prove that the solution to the PDE is stable
with respect to the flux and that, once the control for one
player is fixed, there exists an optimal control for the other player.
The proof of the existence of optimal controls is
based on the classical method of the calculus of variations.
We remark that no total variation estimates are needed for the stability
result and that no Lipschitz continuity hypotheses
on the flux function with respect to
the control variable are needed in the proof of existence of optimal controls.

A natural question for a two-person zero-sum
differential game is whether the value functions coincide,
briefly said as the game has a value.
Under the so called Isaacs condition \cite[Equation~(2.2)]{BCD1997},
it is possible to prove that the value functions coincide,
provided that the uniqueness property for viscosity
solutions of a HJB equation holds; see~\cite{CrandallLionsI}.
Unfortunately, this is not the case in the 
infinite-dimensional setting,
see for example~\cite{CrandallLionsI, CrandallLionsII}.
As a consequence, the question whether the game has
a value is still open in general in this context.

The works of Isaacs~\cite{zbMATH03204219, zbMATH03245077} and
of Pontryagin~\cite{zbMATH03242984} initiated the theory of
two-person zero-sum differential games governed by
ordinary differential equations;
see also the monographs~\cite{zbMATH03365084, zbMATH03245077}
for a complete introduction.
The original motivation was the study
of military problems and the well known pursuit-evasion
game~\cite{zbMATH05877492}
represents the most simple example of a zero-sum differential game in this
subject. The definitions of strategies and of the value were introduced
by Varaiya in~\cite{zbMATH03247458},
see also~\cite{zbMATH03414102, zbMATH03279427}.
The paper~\cite{zbMATH05604590} contains the first proof that the value
function satisfies a Hamilton-Jacobi-Isaacs equation in viscosity sense.
In the context of infinite dimensional Isaacs equation, the literature contains
very few references. We cite here the recent paper~\cite{zbMATH07815234}, where
a zero-sum differential game between a single player and a mass of agents
is considered; see also~\cite{zbMATH07124784}.


The paper is structured as follows. 
In \Cref{sec:Framework} we give the basic definitions and hypotheses
and we state the analytic results. More precisely, we introduce the notions
of non-anticipating strategies and of value functions, we state the
well posedness and stability for system~\eqref{eq:VaxSystem}.
We also provide results about existence of optimal controls in the case
of a single player, and about the fact that the value functions satisfy
a dynamical programming principle and consequently a Hamilton-Jacobi-Isaacs
equation in viscosity sense.
Finally, \Cref{sec:proofs} contains the proofs of
these results.

\section{General framework}
\label{sec:Framework}
This section is devoted to present both the general framework
of the problem, in \Cref{ssec:Definitions}
and the analytical results, in \Cref{ssec:Results}.
Throughout, the notation $\R^+ = [0, \infty)$ is used.

\subsection{Basic definitions and assumptions}
\label{ssec:Definitions}

Here the positive constants $M_1$ and $M_2$ describe the maximum
strength for the two controls $u_1$ and $u_2$. Therefore the sets
\begin{equation}
  \label{eq:admissible-controls}
  \cU_1 \coloneqq \LL\infty\left((0, \infty); [0, M_1]\right)
  \qquad \textrm{ and } \qquad
  \cU_2 \coloneqq \LL\infty\left((0, \infty); [0, M_2]\right)
\end{equation}
denote the admissible controls for the first and the second player,
respectively.
The first equation
in~\eqref{eq:VaxSystem} is a hyperbolic partial differential equation, while the
second one is an ordinary differential equation.
\smallskip

System~\eqref{eq:VaxSystem} is supplemented with the initial conditions
\begin{equation}
  \label{eq:IC}
  \left\{
    \begin{aligned}
      S(0, \xi) & = S_o(\xi), \qquad \xi \in (0,1), \\
      I(0) & = I_o,
    \end{aligned}
  \right.
\end{equation}
where $S_o:[0,1] \to \R^+$ and $I_o \in \R^+$ are given.
\smallskip

For the functions $\alpha$, $f$, and $g$ in~\eqref{eq:VaxSystem}
we introduce the following assumptions.
\begin{enumerate}[label=\bf{($\alpha$)}, ref=\textup{\textbf{($\alpha$)}},
  align =left]
\item \label{hyp:(alpha)} $\alpha \in \Lip(\R^+ ; \R^+)$
  satisfies $\alpha(0) = 0$.
\end{enumerate}

\begin{enumerate}[label=\bf{(F)}, ref=\textup{\textbf{(F)}}, align =left]
    \item \label{hyp:(F)} $f \in \Lip([0, 1] ; [0, 1])$ is non decreasing and
    satisfies $f(0) = 0$, $f(1) = 1$.
\end{enumerate}

\begin{enumerate}[label=\bf{(G)}, ref=\textup{\textbf{(G)}}, align=left]
\item \label{hyp:(G)}
  $g \in \Czero([0,1] \times [0, M_1] \times [0, M_2]; \R)$ satisfies:
  \begin{enumerate}
  \item for all $(u, v) \in [0, M_1] \times [0, M_2]$ the map
    $\xi \mapsto g(\xi, u, v)$ belongs to $\Ck{2}([0, 1] ; \R)$;
  \item for every $(u, v) \in [0, M_1] \times [0, M_2]$,
    $g(0, u, v) = g(1, u, v) = 0$;

  \item there exists $\gamma > 0$ such that
    $\abs{\partial_\xi g\left(\xi, u, v\right)} \le \gamma$ for every
    $\xi \in [0,1]$, $u \in [0, M_1]$, and $v \in [0, M_2]$.
  \end{enumerate}
\end{enumerate}

\begin{enumerate}[label=\bf{(G-1)}, ref=\textup{\textbf{(G-1)}}, align=left]
\item \label{hyp:(G-1)}
  $g$ satisfies \ref{hyp:(G)} and $g(\xi, u, v) = g_1(\xi) + g_2(\xi) u
  + g_3(\xi) v$ for some $g_1, g_2, g_3 \in \Cc2\left([0,1]; \R\right)$.
\end{enumerate}

\begin{remark}
  \label{rmk:alpha}
  The infection rate function $\alpha$, also known as the force of infection,
  models the rate susceptible individuals become infected.
  In classical SIR-type models, the simple linear law
  \begin{equation}
    \label{eq:alpha-1}
    \alpha(I) = \bar \alpha I,
  \end{equation}
  with $\bar \alpha > 0$,
  is considered.
  As pointed out in~\cite{16c9b553-e1f8-3b95-84d9-5427dbebba05},
  a complex phenomenon like the process of infection hardly can be
  represented by a simple law as~\eqref{eq:alpha-1}. The first nonlinear
  models, of type $\bar \alpha I^p$ with $p>1$, for
  the rate of infections were proposed
  in~\cite{capasso1979mathematical, zbMATH03294097,
    16c9b553-e1f8-3b95-84d9-5427dbebba05}.

  More realistic rate functions, also called Holling type functional responses,
  of the type
  \begin{equation}
    \label{eq:alpha-3}
    \alpha(I) = \frac{\bar \alpha I^p}{1 + \beta I^q},
  \end{equation}
  with $p, q > 0$ and $\beta > 0$,
  were proposed in \cite{zbMATH03619693, zbMATH04191458};
  see~\Cref{fig:rate-functions}.
  For more details 
  see also~\cite{zbMATH07621960}. 
    \begin{figure}
    \centering
    \begin{tikzpicture}[line cap=round,line join=round,x=1.cm,y=0.7cm]

      \draw[<-, line width=1.1pt] (0.5, 4.) -- (0.5, 0.);
      \draw[->, line width=1.1pt] (0.5, .5) -- (5., .5);

      \draw[domain=0:10, smooth, variable=\x, color=blue, line
      width=1.2pt] plot ({0.5+4*\x/10}, {0.5+3*\x / (1+\x)});

      \node[anchor=north,inner sep=0] at (0.8,3.9) {$\alpha$};
      \node[anchor=south,inner sep=0] at (5.,.7) {$I$};

      \draw[<-, line width=1.1pt] (0.5+7.5, 4.) -- (0.5+7.5, 0.);
      \draw[->, line width=1.1pt] (0.5+7.5, .5) -- (5.+7.5, .5);

      \draw[domain=0:10, smooth, variable=\x, color=blue, line
      width=1.2pt] plot ({8+4*\x/10}, {0.5 + 5 * \x / (1+\x*\x)});

      \node[anchor=north west,inner sep=0] at (8.1,3.9) {$\alpha$};
      \node[anchor=south,inner sep=0] at (5.+7.5,.7) {$I$};
    \end{tikzpicture}
    \caption{Left: the graph of the infection rate
      function~\eqref{eq:alpha-3} with $p=q=1$.
      Right: the graph of the infection rate function~\eqref{eq:alpha-3} with $p=1$ and $q=2$.}
    \label{fig:rate-functions}
  \end{figure}
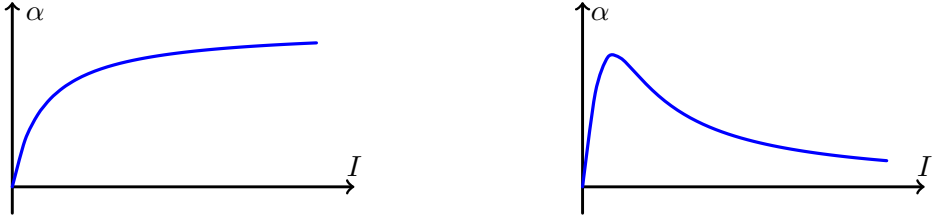
\end{remark}

\begin{remark}
  \label{rmk:fluxG}
  Note that $g$ depends on the tendency variable $\xi$ and also on the two controls
  $u_1$ and $u_2$, which describe the efforts of two competitive players aiming at
  guiding the choices of individuals regarding the vaccination. As an example, one
  player can represent the national health system, which is worried about the possible
  spread of an epidemic, while anti-vaccination activist groups can be identified with
  the other player. A possible choice of the flux function $g$ is given by a
  regularization of the map
  \begin{equation}
    \label{eq:fluxG}
    \left(\xi, u_1, u_2\right) \mapsto
    \xi \left(1 - \xi\right) \left(\xi - \frac{1}{2}
      - \lambda \frac{u_2}{M_2} \1_{A_2}(\xi)
      + (1-\lambda) \frac{u_1}{M_1} \1_{A_1}(\xi)\right),
  \end{equation}
  where the controls $u_1 \in [0, M_1]$ and $u_2 \in [0, M_2]$ act respectively on the
  sets $A_1 \subseteq [0,1]$ and $A_2 \subseteq [0,1]$; see~\Cref{fig:fluxG}. In
  formula~\eqref{eq:fluxG}, $\1_{A_1}$ and $\1_{A_2}$ denote the characteristic
  functions of the sets $A_1$ and $A_2$ respectively. The coefficient
  $\lambda \in \mathopen]0,1\mathclose[$ acts as a weight between the two
  controls. Clearly, formula~\eqref{eq:fluxG} represents in general a discontinuous
  function of $\xi$. In order to avoid difficulties arising in conservation laws with
  discontinuous flux, we prefer to consider the flux function $g$
  in~\eqref{eq:VaxSystem} as a regularization of~\eqref{eq:fluxG}.
  \begin{figure}[!htbp]
    \centering
    \begin{tikzpicture}[line cap=round,line join=round,x=1.cm,y=0.7cm]

      \draw[<-, line width=1.1pt] (0.5, 4.) -- (0.5, 0.); \draw[->, line
      width=1.1pt] (0.5, 2.) -- (5., 2.);

      \draw[domain=0:1, smooth, variable=\x, color=blue, line
      width=1.2pt] plot ({0.5+4*\x}, {2+25*\x*(1-\x)*(\x-0.5)});

      \node[anchor=north,inner sep=0] at (0.8,3.9) {$g$};
      \node[anchor=south,inner sep=0] at (5.,2.1) {$\xi$};

      \draw[<-, line width=1.1pt] (0.5+7.5, 4.) -- (0.5+7.5, 0.);
      \draw[->, line width=1.1pt] (0.5+7.5, 2.) -- (5.+7.5, 2.);

      \draw[domain=0:0.1, smooth, variable=\x, color=dgreen, line
      width=2.pt, dashed] plot ({8+4*\x}, {2+25*\x*(1-\x)*(\x-0.5)});

      \draw[domain=0.1:0.45, smooth, variable=\x, color=dgreen, line
      width=2.pt, dashed] plot ({8+4*\x}, {2+25*\x*(1-\x)*(\x-0.5+0.2)});

      \draw[domain=0.45:0.666, smooth, variable=\x, color=dgreen, line
      width=2.pt, dashed] plot ({8+4*\x}, {2+25*\x*(1-\x)*(\x-0.5-0.3/4+0.2)});

      \draw[domain=0.666:0.85, smooth, variable=\x, color=dgreen, line
      width=2.pt, dashed] plot ({8+4*\x}, {2+25*\x*(1-\x)*(\x-0.5-0.3/4)});

      \draw[domain=0.85:1, smooth, variable=\x, color=dgreen, line
      width=2.pt, dashed] plot ({8+4*\x}, {2+25*\x*(1-\x)*(\x-0.5)});

      \draw[color=blue, line width=1.2pt] (8, 2) [out=-60, in=120]
      to (8.22, 1.4) [out=300, in=200] to (8.7, 1.55)
      [out=20, in=228] to (9.6, 2.6)
      [out=48, in=180] to (9.8, 2.7)
      [out=0, in=210] to (10, 2.8)
      [out=30, in=210] to (10.38, 3.3)
      [out=30, in=110] to (10.6, 3.)
      [out=290, in=210] to (10.9, 2.75)
      [out=30, in=180] to (11.2, 2.9)
      [out=0, in=145] to (11.5, 3.)
      [out=-35, in=110] to (12, 2);

      \draw[line width=1.1pt, color=red] (8+4*0.1, 0.5) -- (8+4*0.666, 0.5);
      \draw[color=red, fill] (8+4*0.1, 0.5) circle (.06cm);
      \draw[color=red, fill] (8+4*0.666, 0.5) circle (.06cm);
      \node[anchor=south, inner sep=0] at (9.5, 0.6) {$A_1$};

      \draw[line width=1.1pt, color=red] (8+4*0.45, 0.3) -- (8+4*0.85, 0.3);
      \draw[color=red, fill] (8+4*0.45, 0.3) circle (.06cm);
      \draw[color=red, fill] (8+4*0.85, 0.3) circle (.06cm);
      \node[anchor=north, inner sep=0] at (10.5, 0.2) {$A_2$};

      \node[anchor=north west,inner sep=0] at (8.1,3.9) {$g$};
      \node[anchor=south,inner sep=0] at (5.+7.5,2.1) {$\xi$};
    \end{tikzpicture}
    \caption{In the left picture the plot of the function in~\eqref{eq:fluxG}
      with controls $u_1 = u_2 = 0$. In this case no regularization is needed.
      \\
      In the right picture the solid-line plot represents a possible
      regularization of the map in~\eqref{eq:fluxG}, represented
      with a dashed line, with controls $u_1 = M_1$ and $u_2 = M_2$ acting
      on the intervals $A_1 = [0.1, 0.666]$ and $A_2 = [0.45, 0.85]$.
    }
    \label{fig:fluxG}
  \end{figure}
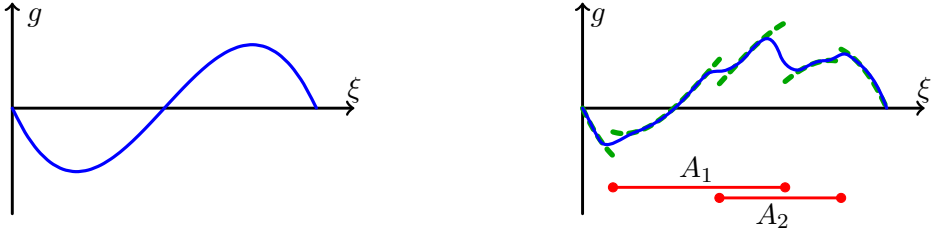
\end{remark}

\subsection{Analytical Results}
\label{ssec:Results}
First we introduce the definition of solution
for system~\eqref{eq:VaxSystem}--\eqref{eq:IC} and we state an existence and uniqueness result for such system.

\begin{definition}
    \label{def:solution}
    Fix $T > 0$, $u_1 \in \LL\infty\left((0, T); [0, M_1]\right)$,
    and $u_2 \in \LL\infty\left((0, T); [0, M_2]\right)$. We say that the couple $(S, I)$ is a solution to~\eqref{eq:VaxSystem}
    on $[0, T]$ with initial condition
    $(S_o, I_o) \in \LL{2}\left((0,1);\R\right) \times \R$ if the following
    conditions are satisfied.
    \begin{enumerate}      
    \item $S \in \Czero([0, T]; \LL{2}\left((0,1); \R^+\right))$
      and $I \in \Czero([0, T]; \R)$.
      
    \item For all test functions $\phi \in \Cc{\infty}((-\infty, T) \times (0, 1); \R)$, it
    holds
    \begin{equation}
        \label{eq:WeakSolution}
        \int_0^{+\infty} \int_0^1
        S (\p_t \phi + g(\xi, u_1(t), u_2(t)) \p_\xi \phi) -
        (f + \alpha(I)) \phi ) \d{\xi} \d{t}
        + \int_0^1 S_o(\xi) \phi(0, \xi) \d{\xi} = 0.
    \end{equation}

  \item For every $t \in [0, T]$
    \begin{equation}
      \label{eq:2}
      I(t) = I_o - \int_0^t \beta I(\tau) \dd \tau + \int_0^t \alpha(I(\tau))
      \left(\int_0^1 S(\tau, \xi) \dd \xi\right) \dd \tau.
    \end{equation}
    \end{enumerate}
\end{definition}

\begin{definition}
  \label{def:solution-global}
  Fix $u_1 \in \cU_1$ and $u_2 \in \cU_2$.
  We say that the couple $(S, I)$ is a solution
  to~\eqref{eq:VaxSystem} on $\R^+$ with initial condition
  $(S_o, I_o) \in \LL{2}\left((0,1);\R^+\right) \times \R^+$ if it is
  a solution to~\eqref{eq:VaxSystem} on $[0, T[$ for every $T > 0$.
\end{definition}

\begin{remark}
  \label{rmk:characteristics-representation}
  With similar reasoning as in~\cite[Theorem~3.2]{Keimer2017}, if
  $(S, I)$ is a solution to~\eqref{eq:VaxSystem} with initial
  condition $(S_o, I_o) \in \LL{2}\left((0,1);\R^+\right) \times \R^+$
  according to \Cref{def:solution-global}, then
  \begin{equation}
    \label{eq:characteristics-formula-general}
    S(t, \xi) = S_o (X(0; t, \xi)) \exp \left(-\int_{0}^{t}
      \left[f(X(s; t, \xi)) + \alpha(I\left(s\right))
        + \p_\xi g(X(s; t, \xi), u_1(s), u_2(s)) \right]\dd{s} \right),
  \end{equation}
  where $X(s; t, \xi)$ solves the Cauchy problem
  \begin{equation}
    \label{eq:charact-curves}
    \left\{
      \begin{array}{l}
        \frac{\dd{}}{\dd s} X(s; t, \xi) = g(X(s; t, \xi), u_1(s), u_2(s)),
        \\
        X(t; t, \xi) = \xi.
      \end{array}
    \right.
  \end{equation}
\end{remark}

The next result deals with existence and uniqueness of solution
to~\eqref{eq:VaxSystem} together with growth estimates.
\begin{theorem}
  \label{th:WellPosedness}
  Assume~\ref{hyp:(alpha)}-\ref{hyp:(F)}-\ref{hyp:(G)}.
  Fix  $u_1 \in \cU_1$ and $u_2 \in \cU_2$.
  Then for all
  $S_o \in \LL{2}((0, 1); \R^+)$ and $I_o \in \R^+$, there exists a unique
  solution to~\eqref{eq:VaxSystem} on $\R^+$ with initial condition
  $(S_o, I_o)$ in the sense of~\Cref{def:solution-global}. Moreover, for all
  $t > 0$,
  \begin{equation}
    \label{eq:WellPosedness}
    \|S(t)\|_{\LL{1}((0,1);\R)}
    \leq \|S_o\|_{\LL{1}((0,1);\R)},
    \quad
    \|S(t)\|_{\LL{2}((0,1);\R)}
      \leq \e^{t \gamma/2} \|S_o\|_{\LL{2}((0,1);\R)}
  \end{equation}
  and
  \begin{equation}
    \label{eq:estimate-sol_I}
      I(t) \leq \|S_o\|_{\LL{1}((0,1);\R)} + I_o.
  \end{equation}
  If moreover $S_o \in \HH1\left((0,1); \R^+\right)$, then for all $t > 0$,
  $S(t) \in \HH1\left((0,1); \R^+\right)$.
\end{theorem}

We state now the well posedness result for~\eqref{eq:VaxSystem}.
\begin{theorem}
  \label{th:Stability}
  Assume~\ref{hyp:(alpha)}-\ref{hyp:(F)}-\ref{hyp:(G)}.  Fix
  $u_1 \in \cU_1$,
  $u_2 \in \cU_2$,
  $S_o, \overline{S}_o \in \LL{2}((0, 1); \R^+)$, and
  $I_o, \overline{I}_o \in \R^+$. Denote by $(S, I)$,
  resp.~$(\overline{S}, \overline{I})$, the global solution to
  \eqref{eq:VaxSystem} with initial condition $(S_o, I_o)$,
  resp.~$(\overline{S}_o, \overline{I}_o)$.  Then for all $t > 0$,
  \begin{equation}
    \label{eq:stability-estimate}
    \begin{split}
    & \quad \norm{S(t)- \overline S(t)}_{\LL1((0,1);\R)}
      + \abs{I(t) - \overline I(t)}
    \\
    & \le \left[\norm{S_o - \overline S_o}_{\LL1((0,1); \R)}
      +  \e^{K_1 \, t} \abs{I_o-\overline I_o}\right]
      \left(1 + t \e^{K_2 t} \mathcal K_t \exp(t \e^{K_2 t} \mathcal K_t)\right),
    \end{split}
  \end{equation}
  where
  \begin{align}
    \label{eq:K12}
    K_1 & \coloneqq \beta + \Lip(\alpha) \norm{S_o}_{\LL1((0,1);\R)},
    &
      K_2 & \coloneqq \Lip(\alpha) \norm{\overline S_o}_{\LL1((0,1);\R)}
  \end{align}
  depend on $\beta$, on $\Lip(\alpha)$, and on the $\LL1$-norm
  of $S_o$ and $\overline S_o$
  respectively, while
  \begin{equation}
    \label{eq:mathcalK}
    \mathcal{K}_t \coloneqq \Lip(\alpha) 
    \left(\norm{\overline S_o}_{\LL1((0,1);\R)} + \e^{K_1 \, t }\abs{\overline I_o} \right)= K_2 + \Lip(\alpha) \abs{\overline I_o} \e^{K_1 \, t}
  \end{equation}
  depends on $K_1$, on $\Lip(\alpha)$, and on $\overline I_o$ .
\end{theorem}

The proofs of~\Cref{th:WellPosedness} and of~\Cref{th:Stability} are
deferred to Section~\ref{ssec:WellPosedness}. They are based on the
separate study of the two equations in~\eqref{eq:VaxSystem} and an
application of Banach Fixed Point Theorem.

\begin{remark}
    When dealing with the unique solution to~\eqref{eq:VaxSystem} with initial condition $(S_o, I_o)$ and controls $u_1, u_2$, it may be relevant to emphasize the dependence on those values. As far as it concerns the controls, we write
\begin{align}
  \label{eq:Y}
  S(t) & = S_{u_1,u_2} (t)
    &
  I(t) & = I_{u_1,u_2} (t)
\end{align}
and we omit to specify the control functions $u_1,u_2$ when they are clear from the context. 
\end{remark}

\begin{remark}
  \Cref{th:Stability} provides only estimates about stability with
  respect to the initial datum. With similar techniques, it is
  possible to consider also stability estimates with respect to the
  flux function $g$ and therefore to the controls. However it would 
  require a bound on the total variation of the
  solution to the PDE; see~\cite{CR2018} for similar results in this direction.
\end{remark}

\bigskip

We now pass to optimal control problems for~\eqref{eq:VaxSystem}.
To this aim, for positive numbers $\kappa, \theta$,
we introduce the cost functional
$\cF:\LL2\left((0,1); \R^+\right) \times \R^+ \times
\cU_1 \times
\cU_2 \to \R$, defined as
\begin{equation}
  \label{eq:CostFunctional}
  \begin{split}
    \cF\left(S_o, I_o, u_1, u_2\right)
    & \coloneqq \kappa \int_{0}^{+\infty} \e^{-\theta t} (I(t) + u_1(t)) \d{t} \\
    & \quad\quad - \int_{0}^{+\infty} \e^{-\theta t} \biggl[ u_2(t)
    + \int_0^1 g(\xi, u_1(t), u_2(t)) \d{\xi} \biggr] \d{t},
  \end{split}
\end{equation}
where $(S, I)$ solves~\eqref{eq:VaxSystem} with initial datum
$(S_o, I_o)$ and controls $u_1$ and $u_2$.
The first term in~\eqref{eq:CostFunctional} is proportional to the number
of infected individuals and to the efforts of the first player, while the
second term is proportional to the efforts of the second player and
to the area of the region where $g$ is positive. Roughly speaking, this area
is related to the change in the leaning towards vaccination.
So the aim of the first, respectively of the second, player is to minimize, respectively maximize, the functional $\cF$ in~\eqref{eq:CostFunctional}. 
The first result deals with the existence of optimal controls.
\smallskip

\begin{theorem}
  \label{th:DecoupledOptimization}
  Assume~\ref{hyp:(alpha)}-\ref{hyp:(F)}-\ref{hyp:(G)} and
  fix $(S_o, I_o) \in \LL{2}((0, 1); \R^+) \times \R^+$.
  Then:
  \begin{enumerate}[label=(\roman*)]
  \item \label{it:u1} for all $u_2 \in \cU_2$,
    the map $u_1 \mapsto \cF\left(S_o, I_o, u_1, u_2\right)$ admits a minimizer in
    $\cU_1$;

  \item \label{it:u2} for all $u_1 \in \cU_1$,
    the map $u_2 \mapsto \cF\left(S_o, I_o, u_1, u_2\right)$ admits a maximizer in
    $\cU_2$.
  \end{enumerate}
\end{theorem}

The proof is deferred to \Cref{ssec:DecoupledOptimization}.
\smallskip

Following~\cite[Chapter VIII]{BCD1997}, we give the
definitions of non-anticipating strategies for both players.

\begin{definition}
    \label{def:strategy}
    A strategy for the first player is a map $\mathcal{S}_1 : \cU_2 \to \cU_1$.
    Moreover, the strategy $\mathcal{S}_1$ is said to be
    non-anticipating if, for all $t > 0$ and for
    all $u_2, \overline{u}_2 \in \cU_2$,
    such that $u_2(s) = \overline{u}_2(s)$ for a.e.~$s \in [0, t]$, then
    \begin{equation*}
      \mathcal{S}_1(u_2) (s) = \mathcal{S}_1(\overline{u}_2) (s)
    \end{equation*}
    for a.e.~$s \in [0, t]$.
    With the symbol $\Gamma$ we denote the set of all the non-anticipating
    strategies for the first player.
  \end{definition}

\begin{definition}
  \label{def:strategy-2}
  A strategy for the second player is a map
  $\mathcal{S}_2 : \cU_1 \to \cU_2$.
  Moreover, the strategy $\mathcal{S}_2$ is said to be
  non-anticipating if, for all $t > 0$ and for
  all $u_1, \overline{u}_1 \in \cU_1$,
  such that $u_1(s) = \overline{u}_1(s)$ for a.e.~$s \in [0, t]$, then
  \begin{equation*}
    \mathcal{S}_2(u_1) (s) = \mathcal{S}_2(\overline{u}_1) (s)
  \end{equation*}
  for a.e.~$s \in [0, t]$.
  With the symbol $\Delta$ we denote the set of all the non-anticipating
  strategies for the second player.
\end{definition}

\smallskip

For readers' convenience, define
\begin{equation}
  \label{eq:CostFunctionalBis1}
    \ell(I, u_1,u_2) \coloneqq 
    \kappa (I + u_1) - u_2 - \int_0^1 g(\xi, u_1, u_2) \dd{\xi},
\end{equation}
so that the cost functional $\cF$, defined in~\eqref{eq:CostFunctional},
can be compactly written as
\begin{equation}
  \label{eq:CostFunctionalBis2}
  \cF(S_o, I_o, u_1, u_2) 
    = \int_{0}^{+\infty} \e^{-\theta t} \ell(I(t), u_1(t), u_2(t)) \dd{t}.
\end{equation}

Now we introduce the lower and upper values for the game.
\begin{definition}
    \label{def:ValuesGame}
    Let $\cF$ be defined by \eqref{eq:CostFunctional}.
    The lower value $V$ and upper value $U$ of the game are respectively
    \begin{equation}
        \label{eq:ValuesGame}
        \begin{split}
          V(S_o, I_o) 
          & \coloneqq \adjustlimits \inf_{\mathcal{S}_1 \in \Gamma} 
          \sup_{u_2 \in \cU_2}
          \cF(S_o, I_o, \mathcal{S}_1(u_2), u_2),
          \\
          U(S_o, I_o) 
          & \coloneqq \adjustlimits\sup_{\mathcal{S}_2 \in \Delta} 
          \inf_{u_1 \in \cU_1}
          \cF(S_o, I_o, u_1, \mathcal{S}_2(u_1)).  
        \end{split}
    \end{equation}
\end{definition}

The lower and upper values satisfy a dynamic programming principle.

\begin{theorem}
    \label{th:DPP}
    Assume~\ref{hyp:(alpha)}-\ref{hyp:(F)}-\ref{hyp:(G)}. Let $\ell$ be as in~\eqref{eq:CostFunctionalBis2}. For all 
    $(S_o, I_o) \in \LL{2}((0, 1); \R) \times \R^+$ and for all $T > 0$, we have
    \begin{equation}
        \label{eq:DPP}
        V(S_o, I_o) =
        \adjustlimits\inf_{\mathcal S_1 \in \Gamma} \sup_{u_2 \in \cU_2} 
        \left\{
          \displaystyle \int_{0}^{T} \e^{-\theta t} \ell\left(
           I_{\cS_1,u_2}(t), \mathcal{S}_1(u_2)(t), u_2(t)\right) \dd{t}
          + V(S_{\cS_1,u_2}(T), I_{\cS_1,u_2}(T)) \e^{-\theta T}
        \right\},
    \end{equation}
    where, with the notation~\eqref{eq:Y}, $(S,I)$ is the solution to~\eqref{eq:VaxSystem} with controls $\mathcal S_1(u_2), u_2$,
    and
    \begin{equation}
      \label{eq:DPP-2}
      U(S_o, I_o) =
      \adjustlimits\sup_{\mathcal S_2 \in \Delta} \inf_{u_1 \in \cU_1} 
      \left\{
        \displaystyle \int_{0}^{T} \e^{-\theta t} \ell\left(
          I_{u_1, \mathcal S_2}(t)
          , u_1(t),\mathcal{S}_2(u_1)(t)
        \right) \dd{t}
        + U(S_{u_1, \mathcal S_2}(T), I_{u_1, \mathcal S_2}(T)) \e^{-\theta T}  
      \right\},
  \end{equation}
  where, with the notation~\eqref{eq:Y}, $(S,I)$ is the solution to~\eqref{eq:VaxSystem} with controls $u_1, \mathcal S_2(u_1)$.
\end{theorem}

The proof is deferred to~\Cref{ssec:Game}. Above, we make a slightly abuse of notation writing only $\mathcal{S}_1$ or $\mathcal{S}_2$ instead of $\mathcal{S}_1(u_2)$ or $\mathcal{S}_2(u_1)$.

{Taking advantage of \Cref{th:DPP}, we prove that $V$ and $U$ both solve 
an Hamilton-Jacobi equation.}
Define, for all $\mathbf{Y} = 
(S, I) \in \HH{1}((0, 1); \R) \times \R$ and 
{$\mathbf{P} = (p, q) \in \LL{2}((0, 1); \R) \times \R$}, the
pre-Hamiltonian
\begin{equation}
    \label{eq:pre-Ham}
    \cH(\mathbf{Y}, \mathbf{P}, u_1, u_2) \coloneqq
    \langle\p_\xi (gS) + (f + \alpha(I)) S , p\rangle_{\LL{2}} 
    + \left(\beta I - \alpha(I) \int_0^1 S (\xi) \dd \xi \right)q 
    - \ell(I, u_1, u_2) 
\end{equation}
and the Hamiltonians
\begin{equation}
    \label{eq:Hamiltonian1}
    \underline {\mathcal H}(\mathbf{Y}, \mathbf{P}) \coloneqq \min_{u_2 \in [0, M_2]} \max_{u_1 \in [0, M_1]}
    \cH(\mathbf{Y}, \mathbf{P}, u_1, u_2)
\end{equation}
and 
\begin{equation}
    \label{eq:Hamiltonian2}
    \overline{\cH}(\mathbf{Y}, \mathbf{P}) \coloneqq \max_{u_1 \in [0, M_1]} \min_{u_2 \in [0, M_2]} 
    \cH(\mathbf{Y}, \mathbf{P}, u_1, u_2).
\end{equation}

\begin{remark}
  Note that in the definition of the Hamiltonian
  functions~\eqref{eq:Hamiltonian1} and~\eqref{eq:Hamiltonian2} we
  decided to consider $S \in \HH1$ and the \textsl{dual variable}
  $p \in \LL2$.  Indeed, there is not any natural duality working in
  our setting, because if $S \in \LL2$ then the derivative
  $\partial_\xi (gS)$ is not well defined in $\LL2$.  Similarly,
  choosing the scalar product of $\HH{1}$, and thus $p \in \HH{1}$,
  would require higher regularity for $S$ to ensure
  $\p_\xi (gS) \in \HH{1}$.
\end{remark}

We give now the definition of viscosity solution for the Hamilton-Jacobi
equation $\theta U + \underline{\cH}(\mathbf{Y}, \nabla U) = 0$.
The definition with the Hamiltonian $\overline{\cH}$,
instead of $\underline{\cH}$,
is completely identical.

\begin{definition}
  Fix $\theta > 0$ and let $\underline{\cH}$ be defined by
  \eqref{eq:Hamiltonian1}. Consider the equation
  \begin{equation}
    \label{eq:HJ}
    \theta U + \underline{\cH}(\mathbf{Y}, \nabla U) = 0,
    \quad \mathbf{Y} \in \HH{1}((0,1); \R) \times \R.
  \end{equation}

  We say that $U\in \Czero(\HH{1}((0,1); \R) \times \R; \R)$ is a
  viscosity sub-solution of \eqref{eq:HJ} if, for every
  $\phi \in \Ck{1}(\LL{2}((0, 1); \R) \times \R; \R)$,
  \[
    \mathbf{Y}_o \in \HH{1}((0,1); \R) \times \R \text{ is a local
      maximum of } U - \phi \implies \theta U(\mathbf{Y}_o) +
    \underline{\cH}(\mathbf{Y}_o, \nabla \phi(\mathbf{Y}_o))
    \leq 0.
  \]

  Similarly, $U\in \Czero(\HH{1}((0,1); \R) \times \R; \R)$ is a
  viscosity super-solution of \eqref{eq:HJ} if, for every
  $\phi \in \Ck{1}(\LL{2}((0, 1); \R) \times \R; \R)$,
  \[
    \mathbf{Y}_o \in \HH{1}((0,1); \R) \times \R \text{ is a local
      minimum of } U - \phi \implies \theta U(\mathbf{Y}_o) +
    \underline{\cH}(\mathbf{Y}_o, \nabla \phi(\mathbf{Y}_o))
    \geq 0.
  \]

  Finally, $U\in \Czero(\HH{1}((0,1); \R) \times \R; \R)$ is viscosity
  solution of \eqref{eq:HJ} if it is both a viscosity sub-solution and
  a viscosity super-solution.
\end{definition}

The value functions $U$ and $V$, defined in~\eqref{eq:ValuesGame}, are viscosity solutions to some
Hamilton-Jacobi equations.
\begin{theorem}
  \label{th:ValueViscosity1}
  Assume~\ref{hyp:(alpha)}-\ref{hyp:(F)}-\ref{hyp:(G)}.
  Let $\underline{\cH}$ and $\overline{\cH}$
  be defined by \eqref{eq:Hamiltonian1} and by \eqref{eq:Hamiltonian2}. Then 
  \begin{enumerate}[label=(\roman*)]
  \item  $V$, defined by \eqref{eq:ValuesGame}, is a viscosity
    solution to~\eqref{eq:HJ} with Hamiltonian $\underline{\cH}$.
    
  \item $U$, defined by \eqref{eq:ValuesGame}, is a viscosity solution
    to~\eqref{eq:HJ} with Hamiltonian $\overline{\cH}$. 
  \end{enumerate}
\end{theorem}

\section{Proofs}
\label{sec:proofs}
This section contains the proofs of the analytic results presented
in \Cref{ssec:Results}.

\subsection{Proof of Theorem~\ref{th:WellPosedness}}
\label{ssec:WellPosedness}

The proof of \Cref{th:WellPosedness} consists in studying the
two equations of \eqref{eq:VaxSystem} separately and then applying
a Fixed Point Theorem.

\subsubsection{The Controlled Partial Differential Equation}

Here, let us consider the Cauchy Problem for the
partial differential equation:
\begin{equation}
  \label{eq:VaxPDE}
  \left\{
    \begin{array}{l@{\qquad}l}
      \p_t S + \p_\xi (g(\xi, u_1(t), u_2(t)) S) = -(f + \alpha (I)) S,
      & t > 0, \, \xi \in (0,1),
      \\
      S(0, \xi) = S_o(\xi),
      & \xi \in (0,1),
    \end{array}
  \right.
\end{equation}
where $S_o \in \LL2\left((0,1); \R\right)$,
$\alpha$ satisfies~\ref{hyp:(alpha)}, $f$ satisfies~\ref{hyp:(F)},
$I \in \Czero(\R^+; \R^+)$,
and $u_1$ and $u_2$ are control functions.
Let us introduce the definition of solution for~\eqref{eq:VaxPDE}.
\begin{definition}
  \label{def:solution-only-PDE}
  We say that the function $S$ is a solution to~\eqref{eq:VaxPDE} with
  initial condition $S_o \in \LL{2}\left((0,1);\R\right)$ if the
  following conditions are satisfied.
  \begin{enumerate}
  \item $S \in \Czero(\R^+; \LL{2}\left((0,1);
      \R\right))$.
      
  \item For all test functions
    $\phi \in \Cc{\infty}(\R \times (0, 1); \R)$, it holds
    \begin{equation}
      \label{eq:WeakSolution_PDE}
      \int_0^{+\infty} \int_0^1
      S (\p_t \phi + g(\xi, u_1(t), u_2(t)) \p_\xi \phi) -
      (f + \alpha(I)) \phi ) \d{\xi} \d{t}
      + \int_0^1 S_o(\xi) \phi(0, \xi) \d{\xi} = 0.
    \end{equation}
  \end{enumerate}
\end{definition}

Here we state the well posedness for system~\eqref{eq:VaxPDE}.
\begin{lemma}
  \label{lmm:WellPosedness1}
  Assume~\ref{hyp:(alpha)}, \ref{hyp:(F)} and~\ref{hyp:(G)}.  Fix
  {$I \in \Czero(\R^+; \R^+)$}, and the control
  functions $u_1 \in \cU_1$ and
  $u_2 \in \cU_2$.  Then, for all
  $S_o \in \LL{2}((0, 1); \R)$, there exists a unique solution $S$
  to~\eqref{eq:VaxPDE}, in the sense of~\Cref{def:solution-only-PDE}.
  Moreover:
  \begin{enumerate}[label=(\roman*)]
  \item if $S_o \geq 0$, then for all $t > 0$, $S(t) \geq 0$ ;

  \item for all $t > 0$, we have that
  \begin{equation}
    \label{eq:WellPosedness-L1}
    \|S(t)\|_{\LL{1}\left((0,1); \R\right)} \leq
    \|S_o\|_{\LL{1}\left((0,1); \R\right)},
  \end{equation}
  and
  \begin{equation}
    \label{eq:WellPosedness-L2}
    \|S(t)\|_{\LL{2}\left((0,1); \R\right)} \leq
    \e^{t\gamma/2} \|S_o\|_{\LL{2}\left((0,1); \R\right)},
  \end{equation}
  where $\gamma$ is defined in~\ref{hyp:(G)}.
\end{enumerate}
\end{lemma}

\begin{proof}
  For readers' convenience, set the flux function
  \begin{equation}
    \label{eq:G-flux}
    G(\xi, t)\coloneqq g(\xi, u_1(t), u_2(t)).
  \end{equation}

  We split the proof into several parts.

  \textbf{Existence of $\LL\infty$ solutions.}
  Let $(S_o^{(n)})_n$ be a sequence of functions in
  $\Cc{\infty}((0, 1); \R)$ such that
  \begin{equation*}
    \lim_{n \to +\infty} \norm{S_o^{(n)} - S_o}_{\LL2\left((0,1); \R\right)} = 0.
  \end{equation*}
  For all $n \in \N$, the Cauchy problem \eqref{eq:VaxPDE},
  with initial condition $S_o^{(n)}$, admits a unique classical solution
  $S_n$, which can be obtained using the method of characteristics,
  for instance. Here~\ref{hyp:(G)}, together with the regularity of the controls
  $u_1$ and $u_2$, permits to have globally defined characteristics curves.
  Since, for every $n \in \N$, $S_n$ is a classical solution,
  then $S_n$ satisfies condition~\eqref{eq:WeakSolution_PDE}
  for every
  test function $\phi$.

  Multiplying the first equation in~\eqref{eq:VaxPDE} by $S_n$ and
  integrating, we get that, for all $t > 0$,
  \begin{align*}
    \frac{\d{}}{\d{t}} \|S_n(t)\|_{\LL{2}\left((0,1); \R\right)}^2
    & = 2 \int_0^1 S_n  \, \partial_t S_n \dd\xi
      = -2 \int_0^1 \left(f + \alpha(I)\right)S_n^2 \dd\xi
      -2 \int_0^1 \partial_\xi \left(G(\xi, t) S_n\right) S_n \dd\xi
    \\
    & \le
      -2 \int_0^1 \partial_\xi \left(G(\xi, t) S_n\right) S_n \dd\xi
      = 2 \int_0^1 G(\xi, t) S_n \partial_\xi S_n \dd\xi
      = \int_0^1 G(\xi, t) \partial_\xi S_n^2 \dd\xi
    \\
    & = - \int_0^1 \p_\xi G(\xi, t) S_n^2 \d{\xi}
    \leq \gamma \|S_n(t)\|_{\LL{2}\left((0,1); \R\right)}^2,
  \end{align*}
  where we used~\ref{hyp:(alpha)}, \ref{hyp:(F)}, and~\ref{hyp:(G)}.
  Here the constant $\gamma$ is defined in~\ref{hyp:(G)}.
  Gronwall Lemma ensures that, for every $n \in \N$ and $t > 0$,
  \begin{equation}
    \label{eq:S_n-L2}
    \|S_n(t)\|_{\LL{2}\left((0,1); \R\right)}
    \leq \e^{t \gamma/2} \|S_o^{(n)}\|_{\LL{2}\left((0,1); \R\right)} .
  \end{equation}
  We deduce that, for all $T > 0$, $(S_n)_n$ is bounded in
  $\LL{\infty}((0, T); \LL{2}\left((0,1); \R\right))$.
  Therefore, for all $T > 0$,
  there exists $S_T \in \LL{\infty}((0, T); \LL{2}\left((0,1); \R\right))$
  such that a subsequence of $(S_n)_n$ converges to $S_T$
  in the $\LL{\infty}$-weak* sense.
  By a standard diagonal process, we can construct $S$ such that
  there exists a subsequence of $(S_n)_n$ that weakly* converges to $S$
  in $\LL{\infty}((0, \infty); \LL{2}\left((0,1); \R\right))$.
  By linearity, $S$ satisfies~\eqref{eq:WeakSolution_PDE}
  for every test function.
  Passing to the limit as $n \to +\infty$ in~\eqref{eq:S_n-L2},
  we deduce~\eqref{eq:WellPosedness-L2}.

  \textbf{Uniqueness.} By linearity, it is sufficient to prove that if
  $S_o \equiv 0$, then any weak solution is also zero.
  We follow here a similar technique as in~\cite{zbMATH05837938}.
  Fix $\psi \in \Cc{\infty}(\R^+ \times (0, 1); \R)$.
  We claim that
  \begin{equation*}
    \int_0^{+\infty} \in,t_0^1 S(t, \xi) \psi(t, \xi) \d{\xi} \d{t} = 0
  \end{equation*}
  so that $S(t, \xi) = 0$ for a.e.~$(t, \xi)$ by the arbitrariness of $\psi$;
  hence uniqueness holds.

  To this aim, there exists
  $\phi \in \Cc{\infty}(\R^+ \times (0, 1); \R)$ such that
  \begin{equation}
    \label{eq:adjoint}
    \p_t \phi + G(\xi, t) \p_\xi \phi - (f + \alpha (I))\phi = \psi.
  \end{equation}
  Therefore, using~\eqref{eq:adjoint} and~\eqref{eq:WeakSolution_PDE},
  we deduce that
  \begin{align*}
    \int_0^{+\infty} \int_0^1 S(t, \xi) \psi(t, \xi) \d{\xi} \d{t}
    & =
      \int_0^{+\infty} \int_0^1
      S (\p_t \phi + G(\xi, t) \p_\xi \phi) -
      (f + \alpha (I)) \phi ) \d{\xi} \d{t}
    \\
    & = - \int_0^1 S_o(\xi) \phi(0, \xi) \dd\xi = 0,
  \end{align*}
  since $S_o \equiv 0$. Thus uniqueness holds.
  \smallskip

  \textbf{Positivity.} The method of characteristics implies that the
  unique solution to the Cauchy problem~\eqref{eq:VaxPDE} can be written
  in the form
  \begin{equation}
    \label{eq:characteristics-formula}
    S(t, \xi) = S_o (X(0; t, \xi)) \exp \left(-\int_{0}^{t}
      \left[f(X(s; t, \xi)) + \alpha(I\left(s\right))
        + \p_\xi G(X(s; t, \xi), s) \right]\dd{s} \right),
  \end{equation}
  where $X(s; t, \xi)$ solves the Cauchy problem
  \begin{equation*}
    \left\{
      \begin{array}{l}
        X'(s; t, \xi) = G(X(s; t, \xi), s),
        \\
        X(t; t, \xi) = \xi.
      \end{array}
    \right.
  \end{equation*}
  In the case $S_o(\xi) \ge 0$ for a.e.~$\xi \in (0,1)$,
  formula~\eqref{eq:characteristics-formula} implies that, for every $t>0$,
  $S(t, \xi) \ge 0$ for a.e.~$\xi \in (0, 1)$.
  \smallskip

  \textbf{$\LL1$ estimate.}
  For all $t > 0$, by assumptions~\ref{hyp:(alpha)} and~\ref{hyp:(F)}, formula~\eqref{eq:characteristics-formula}
  and the change of variable $y = X(0; t, \xi)$, we get
  \begin{align*}
    &
      \|S(t)\|_{\LL{1}\left((0, 1); \R\right)}
    \\
    = & \int_0^1 |S_o(X(0; t, \xi))| \exp \left(-\int_{0}^{t}
      \left[f(X(s; t, \xi)) + \alpha\left(I\left(s\right)\right)
      + \p_1 G(X(s, t, \xi), s) \right]\d{s} \right) \d{\xi}
    \\
    = & \int_{0}^1 |S_o(y)| \exp \left(-\int_{0}^{t} \left[f(X(s; 0, y))
      + \alpha\left(I\left(s\right)\right)\right]
      \d{s}\right)\d{y}
      \leq \|S_o\|_{\LL{1}\left((0,1); \R\right)},
  \end{align*}
  proving~\eqref{eq:WellPosedness-L1}.
  \smallskip

  \textbf{Continuity in time.}
  We exploit an idea from~\cite[Lemma~2.11]{Keimer2017}.
  Fix $t_1 \ge 0$ and consider $t_2 \ge 0$. Without loss of generality we
  assume that $t_1 < t_2$. By the method of characteristics we have:
  \begin{displaymath}
    S(t_2, \xi) - S(t_1, \xi)
    =
    S\left(t_1, X(t_1; t_2, \xi)\right)\mathcal{E}(t_1, t_2, \xi)
    - S(t_1, \xi),
  \end{displaymath}
  where
  \begin{equation}
    \label{eq:exp}
    \mathcal{E}(t_1, t_2, \xi) =
    \exp \left(-\int_{t_1}^{t_2}
      \left[f(X(s; t_2, \xi)) + \alpha(I\left(s\right))
        + \p_\xi G(X(s; t_2, \xi), s) \right]\d{s} \right) .
  \end{equation}
  In particular,
  adding and subtracting $S(t_1, \xi) \, \mathcal{E}(t_1, t_2, \xi)$,
  it yields
  \begin{equation}
    \label{eq:t1}
    \begin{split}
      \abs{S(t_2, \xi) - S(t_1, \xi)}^2 
      & = \abs{
        \left(
          S\left(t_1, X(t_1; t_2, \xi)\right) - S(t_1, \xi)
        \right) \mathcal{E}(t_1, t_2, \xi)
        + S(t_1,\xi) \left(\mathcal{E}(t_1, t_2, \xi) -1\right)
      }^2
      \\
      & \le \abs{S\left(t_1, X(t_1; t_2, \xi)\right) - S(t_1, \xi)}^2
      \left(\mathcal{E}(t_1, t_2, \xi)\right)^2
      \\
      & \quad + \left(S(t_1,\xi)\right)^2 \abs{\mathcal{E}(t_1, t_2, \xi) -1}^2 
      \\
      & \quad 
      + 2 \abs{S\left(t_1, X(t_1; t_2, \xi)\right) - S(t_1, \xi)} \cdot
      \abs{S(t_1,\xi)} \, \mathcal{E}(t_1, t_2, \xi) \,
      \abs{\mathcal{E}(t_1, t_2, \xi) -1}.
    \end{split}
  \end{equation}
  By the positivity of $f$ and $\alpha$, and exploiting also~\ref{hyp:(G)},
  we have that
  \begin{equation}
    \label{eq:t1_bound}
    \begin{split}
      & \int_0^1 \abs{S\left(t_1, X(t_1; t_2, \xi)\right) - S(t_1, \xi)}^2
      \left(\mathcal{E}(t_1, t_2, \xi)\right)^2 \dd\xi
      \\
      \le
      & \e^{2 \, \gamma (t_2-t_1)}
      \int_0^1  
      \abs{S\left(t_1, X(t_1; t_2, \xi)\right) - S(t_1, \xi)}^2
      \dd{\xi},
    \end{split}
  \end{equation}
  and the right hand side goes to zero as $t_2 \to t_1$
  by a similar reasoning as in~\cite[Lemma~4.3]{Brezis}.

  The definition~\eqref{eq:exp} of $\mathcal{E}$, together
  with~\ref{hyp:(alpha)}, \ref{hyp:(F)} and~\ref{hyp:(G)}, yields
  \begin{equation}
    \label{eq:E-1infty}
    \norm{\mathcal{E}(t_1, t_2, \cdot) - 1}_{\LL\infty((0,1); \R)}
    \le (t_2-t_1)
    \e^{\gamma(t_2-t_1)}\left(\norm{f}_{\LL\infty((0,1);\R)}
      + \Lip(\alpha) \norm{I}_{\LL\infty([t_1,t_2]; \R)} +  \gamma\right),
  \end{equation}
  so that
  \begin{equation}
    \label{eq:t2_bound}
    \begin{split}
      \int_0^1 \left(S(t_1,\xi)\right)^2
      \abs{\mathcal{E}(t_1, t_2, \xi) -1}^2\dd\xi
      &
      \le \norm{S(t_1)}^2_{\LL2((0,1); \R)}
      \norm{\mathcal{E}(t_1, t_2, \cdot) -1}^2_{\LL\infty((0,1); \R)}
      \\
      & \le (t_2-t_1)^2
      \e^{2\gamma(t_2-t_1)} \norm{S(t_1)}^2_{\LL2((0,1); \R)}
      \\
      & \quad \times
      \left(\norm{f}_{\LL\infty((0,1);\R)} 
        + \Lip(\alpha) \norm{I}_{\LL\infty([t_1,t_2]; \R)} +  \gamma\right)^2
    \end{split}
  \end{equation}
  and the right hand side goes to zero as $t_2 \to t_1$,
  since $\norm{S(t_1)}_{\LL2((0,1); \R)}$
  is bounded by~\eqref{eq:WellPosedness-L2}.

  Finally, using Young's inequality, the positivity of $f$ and $\alpha$,
  and exploiting also~\ref{hyp:(G)}, we have
  \begin{equation}
    \label{eq:t3_bound}
    \begin{split}
      & 2 \int_0^1 \abs{S\left(t_1, X(t_1; t_2, \xi)\right) - S(t_1, \xi)} \cdot 
      \abs{S(t_1,\xi)} \, \mathcal{E}(t_1, t_2, \xi) \,
      \abs{\mathcal{E}(t_1, t_2, \xi) -1}\dd\xi
      \\
      \le &
      \norm{\mathcal{E}(t_1, t_2, \cdot) - 1}_{\LL\infty((0,1); \R)}
      \left(
        \int_0^1 
        \abs{S\left(t_1, X(t_1; t_2, \xi)\right) - S(t_1, \xi)}^2 \dd\xi
        +
        \int_0^1 \abs{S(t_1,\xi)}^2 \left(\mathcal{E}(t_1, t_2, \xi)\right)^2
        \dd{\xi}
      \right)
      \\
      \le & 
      \norm{\mathcal{E}(t_1, t_2, \cdot) - 1}_{\LL\infty((0,1); \R)}
      \left(
        \int_0^1 
        \abs{S\left(t_1, X(t_1; t_2, \xi)\right) - S(t_1, \xi)}^2 \dd\xi
        +
        \e^{2\, \gamma (t_2-t_1)} \norm{S(t_1)}^2_ {\LL2((0,1);\R)}
      \right),
    \end{split}
  \end{equation}
  and the right hand side goes to zero as $t_2 \to t_1$
  by~\eqref{eq:E-1infty}, \cite[Lemma~4.3]{Brezis}, and the boundedness
  of $\norm{S(t_1)}_{\LL2((0,1); \R)}$ due to~\eqref{eq:WellPosedness-L2}.

  Integrating~\eqref{eq:t1} over $\xi \in [0,1]$ and
  using~\eqref{eq:t1_bound}, \eqref{eq:t2_bound}, and~\eqref{eq:t3_bound}
  we deduce that
  \begin{equation*}
    \lim_{t_2 \to t_1} \norm{S(t_2) - S(t_1)}_{\LL2((0,1);\R)} = 0,
  \end{equation*}
  proving that
  $S \in \Czero(\R^+; \LL2((0,1);\R))$.
\end{proof}

We provide now a stability result for~\eqref{eq:VaxPDE}. For simplicity,
we call ``weak solution to~\eqref{eq:VaxPDE} with data $(S_o,I)$'' a
weak solution to~\eqref{eq:VaxPDE}  with initial datum $S_o$ and
given function $I$.
\begin{lemma}
  \label{lmm:WellPosedness2}
  Assume~\ref{hyp:(alpha)}, \ref{hyp:(F)} and~\ref{hyp:(G)}.
  Fix $u_1 \in \cU_1$, $u_2 \in \cU_2$, 
  $I, \overline{I} \in\Czero(\R^+; \R^+)$ and 
  $S_o, \overline{S}_o \in \LL{2}((0, 1); \R)$. Denote by $S$,
  resp.~$\overline{S}$, the weak solution to~\eqref{eq:VaxPDE} with
  data $(S_o, I)$, resp.~$(\overline{S}_o, \overline{I})$. Then for
  all $t > 0$,
  \begin{equation}
    \label{eq:WellPosedness2a}
    \norm{S(t) - \overline{S}(t)}_{\LL{1}\left((0,1); \R\right)}
    \leq \norm{S_o - \overline{S}_o}_{\LL{1}\left((0,1);\R\right)}
    + \Lip\left(\alpha\right) \norm{\overline{S}_o}_{\LL{1}\left((0,1);\R\right)}
    \norm{I - \overline I}_{\LL1\left((0,t); \R\right)}
  \end{equation}
  and
  \begin{equation}
    \label{eq:WellPosedness2b}
    \norm{S(t) - \overline{S}(t)}_{\LL{2}((0,1);\R)} \le \e^{t \gamma /2}
     \left(
    \norm{S_o - \overline S_o}_{\LL2\left((0,1); \R\right)}
    + \Lip(\alpha) \sqrt{t}
    \norm{\overline S_o}_{\LL2\left((0,1); \R\right)}
    \norm{I - \overline I}_{\LL2\left((0,t); \R\right)}  \right),
  \end{equation}
  where $\gamma$ is defined in~\ref{hyp:(G)}.

\end{lemma}

\begin{proof}
  Let $\Sigma$ be the unique weak solution to~\eqref{eq:VaxPDE}
  with data $(\overline{S}_o, I)$.
  Since $S - \Sigma$ solves~\eqref{eq:VaxPDE} with initial data
  $S_o - \overline{S}_o$, we
  deduce from~\Cref{lmm:WellPosedness1} that, for all $t > 0$,
  \begin{equation}
    \label{eq:1}
    \|S(t) - \Sigma(t)\|_{\LL{1}\left((0,1); \R\right)}
    \leq \|S_o - \overline{S}_o\|_{\LL{1}\left((0,1); \R\right)}
  \end{equation}
  and
  \begin{equation}
    \label{eq:3}
    \|S(t) - \Sigma(t)\|_{\LL{2}\left((0,1);\R\right)}
    \leq \e^{t\gamma/2} \|S_o - \overline{S}_o\|_{\LL{2}\left((0,1); \R\right)}.
  \end{equation}
  The function $\delta \coloneqq \Sigma-\overline{S}$ is the unique weak solution to
  \begin{equation*}
    \left\{
      \begin{aligned}
        \p_t \delta + \p_\xi (G(\xi, t) \delta)
        & = -(f + \alpha (I)) \delta - \left(\alpha(I) - \alpha(\overline{I})
        \right) \overline{S}
        \\
        \delta(0, \xi) & = 0,
      \end{aligned}
    \right.
  \end{equation*}
  and therefore, can be written using characteristics as
  \begin{equation}
    \label{eq:delta-expression}
    \begin{split}
      \delta(t, \xi)
      & = - \int_0^t (\alpha(I(\tau)) - \alpha(\overline{I}(\tau)))
      \, \overline{S}(\tau, X(\tau; t, \xi))
      \\
      & \qquad \times \exp\left( -\int_{\tau}^t
        \left[f(X(s; t, \xi)) + \alpha(I(s))+ \p_1 G(X(s; t, \xi), s) \right]
        \dd{s} \right) \dd{\tau},
    \end{split}
  \end{equation}
  where $X = X(\tau; t, \xi)$ and $G$ are defined as in the proof
  of~\Cref{lmm:WellPosedness1}.
  Using~\ref{hyp:(F)} and~\ref{hyp:(alpha)}, the
  change of variable $y = X(\tau; t, \xi)$, and
  formula~\eqref{eq:WellPosedness-L1},
  for all $t > 0$, we obtain that
  \begin{align*}
    \|\delta(t)\|_{\LL{1}\left((0,1); \R\right)}
    & \leq \int_0^1 \int_0^t |\alpha (I(\tau)) - \alpha (\overline{I}(\tau))| \cdot
      |\overline{S}(\tau, X(\tau, t, \xi))|
      \exp\left( -\int_{\tau}^t
      \p_1 G(X(s; t, \xi), s)
      \dd{s} \right) \dd{\tau} \dd\xi
    \\
    & \le \Lip(\alpha)
      \int_0^t |I(\tau) - \overline{I}(\tau)|
      \int_0^1 |\overline{S}(\tau, y)|
      \d{y} \d{\tau}
    \\
    & = \Lip(\alpha)
      \int_0^t |I(\tau) - \overline{I}(\tau)|\,
      \|\overline{S}(\tau)\|_{\LL1\left((0,1); \R\right)}
      \dd{\tau}
    \\
    & \le \Lip(\alpha) \|\overline{S}_o\|_{\LL1\left((0,1); \R\right)}
      \int_0^t |I(\tau) - \overline{I}(\tau)|\, \dd{\tau}.
  \end{align*}
  Hence, using~\eqref{eq:1}, we conclude that, for all $t > 0$,
  \begin{align*}
    \norm{S(t) - \overline S(t)}_{\LL1\left((0,1); \R\right)}
    & \le \norm{S(t) - \Sigma(t)}_{\LL1\left((0,1); \R\right)}
      + \norm{\Sigma(t) - \overline S(t)}_{\LL1\left((0,1); \R\right)}
    \\
    & \le \norm{S_o - \overline S_o}_{\LL1\left((0,1); \R\right)}
      + \Lip(\alpha) \|\overline{S}_o\|_{\LL1\left((0,1); \R\right)}
      \int_0^t |I(\tau) - \overline{I}(\tau)|\, \dd{\tau},
  \end{align*}
  proving~\eqref{eq:WellPosedness2a}.

  Moreover, using~\ref{hyp:(alpha)} and~\ref{hyp:(F)}, from~\eqref{eq:delta-expression} we deduce that
  \begin{align*}
    |\delta(t, \xi)|
    & \le \int_0^t \abs{\alpha\left(I(\tau)\right)
      - \alpha\left(\overline I(\tau)\right)} \cdot
      \abs{\overline S\left(\tau, X(\tau; t, \xi)\right)}
      \exp\left(-\int_{\tau}^t \p_1 G\left(X(s; t, \xi), s\right) \dd s \right)
      \dd \tau
    \\
    & \le \Lip(\alpha) \int_0^t \abs{I(\tau) - \overline I(\tau)} \cdot
      \abs{\overline S\left(\tau, X(\tau; t, \xi)\right)}
      \exp\left(-\int_{\tau}^t \p_1 G\left(X(s; t, \xi), s\right) \dd s \right)
      \dd \tau,
  \end{align*}
  so that, by Hölder inequality,
  \begin{equation*}
    |\delta(t, \xi)|^2 \leq \left(\Lip(\alpha)\right)^2 t \int_0^t |I(\tau) -
    \overline{I}(\tau)|^2 \cdot |\overline{S}(\tau, X(\tau; t, \xi))|^2
    \exp\left( -\int_{\tau}^t 2 \p_1 G(X(s; t,
      \xi), s)\dd{s} \right) \dd{\tau}.
  \end{equation*}
  Integrating over $\xi \in [0,1]$
  and using the change of variable $y = X(\tau; t, \xi)$
  and~\eqref{eq:WellPosedness-L2}, we obtain that,
  for all $t > 0$,
  \begin{align*}
    \|\delta(t)\|_{\LL{2}\left((0,1); \R\right)}^2
    & \leq \left(\Lip(\alpha)\right)^2 t \int_0^t |I(\tau) -
      \overline{I}(\tau)|^2 \int_0^1 |\overline{S}(\tau, y)|^2
      \e^{(t-\tau)\gamma} \d{y} \d{\tau}
    \\
    & \le \left(\Lip(\alpha)\right)^2 t \int_0^t |I(\tau) -
      \overline{I}(\tau)|^2 \e^{(t-\tau)\gamma}
      \norm{\overline S(\tau)}_{\LL2\left([0,1]; \R\right)}^2 \d{\tau}
    \\
    & \le \left(\Lip(\alpha)\right)^2 t
      \norm{\overline S_o}_{\LL2\left((0,1); \R\right)}^2 \int_0^t |I(\tau) -
      \overline{I}(\tau)|^2 \e^{\tau \, \gamma}
       \d{\tau}
    \\
    & \le \left(\Lip(\alpha)\right)^2 t
      \norm{\overline S_o}_{\LL2\left((0,1); \R\right)}^2
      \e^{t \, \gamma} \norm{I - \overline I}_{\LL2\left((0,t); \R\right)}^2.
  \end{align*}
  Therefore, for all $t > 0$, using~\eqref{eq:3}, we conclude that
  \begin{align*}
    \norm{S(t) - \overline S(t)}_{\LL2\left((0,1); \R\right)}
    & \le \norm{S(t) - \Sigma(t)}_{\LL2\left((0,1); \R\right)}
      + \norm{\Sigma(t) - \overline S(t)}_{\LL2\left((0,1); \R\right)}
    \\
    & \le \e^{t \gamma /2} \norm{S_o \!-\! \overline S_o}
      _{\LL2\left((0,1); \R\right)}
      \!+ \left(\Lip(\alpha)\right) \sqrt{t}
      \norm{\overline S_o}_{\LL2\left((0,1); \R\right)}
      \e^{t\gamma/2} \norm{I - \overline I}_{\LL2\left((0,t); \R\right)},
  \end{align*}
  proving~\eqref{eq:WellPosedness2b}.
\end{proof}

\begin{remark}
  Note that the estimates~\eqref{eq:WellPosedness2a}
  and~\eqref{eq:WellPosedness2b} in the statement of~\Cref{lmm:WellPosedness2}
  are symmetric with respect to the initial data $S_o$ or $\overline S_o$.
  Hence~\eqref{eq:WellPosedness2a} can be replace by
  \begin{align*}
    \norm{S(t) - \overline{S}(t)}_{\LL{1}\left((0,1); \R\right)}
    & \leq \norm{S_o - \overline{S}_o}_{\LL{1}\left((0,1);\R\right)}
    \\
    & \quad
      + \Lip\left(\alpha\right) \min\left\{
      \norm{{S}_o}_{\LL{1}\left((0,1);\R\right)},
      \norm{\overline{S}_o}_{\LL{1}\left((0,1);\R\right)}\right\}
      \norm{I - \overline I}_{\LL1\left((0,t); \R\right)}
  \end{align*}
  while~\eqref{eq:WellPosedness2b} can be replaced by
  \begin{align*}
    \norm{S(t) - \overline{S}(t)}_{\LL{2}((0,1);\R)}
    & \le \e^{t \gamma /2}
      \norm{S_o - \overline S_o}_{\LL2\left((0,1); \R\right)}
    \\
    & \quad
      + \e^{t \gamma /2}\Lip(\alpha) \sqrt{t}
      \min\left\{\norm{S_o}_{\LL2\left((0,1); \R\right)},
      \norm{\overline S_o}_{\LL2\left((0,1); \R\right)}
      \right\}
      \norm{I - \overline I}_{\LL2\left((0,t); \R\right)}.
  \end{align*}
\end{remark}

\subsubsection{The Ordinary Differential Equation}

In this part we recall classical results about the Cauchy problem
\begin{equation}
  \label{eq:VaxODE}
  \left\{
    \begin{array}{l}
      \displaystyle
      \dot I(t) = -\beta I(t) + \alpha (I(t)) \int_0^1 S(t, \xi) \d{\xi},
      \\
      I(0) = I_o,
    \end{array}
  \right.
\end{equation}
where $\beta>0$ and $I_o \in \R$.

\begin{lemma}
  \label{lmm:WellPosedness3}
  Assume~\ref{hyp:(alpha)}.
  Let $S \in \LL{\infty}((0, \infty); \LL{2}((0, 1); \R))$.
  Then, for all $I_o \in \R$,
  the Cauchy problem~\eqref{eq:VaxODE}
  admits a unique solution $I \in \Czero(\R^+; \R)$. Moreover,
  \begin{enumerate}[label=(\roman*), ref=\textsl{(\roman*)}]
    \item \label{it:ODE-i} if $I_o \geq 0$, then $I(t) \geq 0$ for all $t > 0$;
    \smallskip

    \item \label{it:ODE-ii} for all $t > 0$,
    \begin{equation}
        \label{eq:WellPosedness3}
        |I(t)| \leq |I_o| \e^{K_t t},
    \end{equation}
    where
    $K_t \coloneqq \beta + \Lip(\alpha) \norm{S}_{\LL\infty((0, t); \LL1((0,1);\R))}$.
  \end{enumerate}
\end{lemma}

\begin{proof}
  The proof of the well-posedness as well
  as Item \emph{(i)} are standard, hence we omit them.
  Let us precise here the importance of the fact that $\alpha(0) = 0$.
  Inequality~\eqref{eq:WellPosedness3} comes from an application of
  Gronwall Lemma since for all $t > 0$,
  \begin{equation*}
    |I(t)| \leq |I_o| + \left(\beta + \Lip(\alpha)
      \norm{S}_{\LL\infty((0, t); \LL1((0,1);\R))}\right) \int_0^t
    \abs{I(\tau)} \d{\tau}.
  \end{equation*}
\end{proof}

We continue with a stability result for the Cauchy
problem~\eqref{eq:VaxODE}. For simplicity, we call ``solution
to~\eqref{eq:VaxODE} with data $(S,I_o)$'' a solution
to~\eqref{eq:VaxODE} with initial datum $I_o$ and given function $S$.
\begin{lemma}
  \label{lmm:WellPosedness4}
  Assume~\ref{hyp:(alpha)}.
  Fix $S, \overline{S} \in \LL{\infty}((0, \infty); \LL{2}((0, 1); \R))$ and
  the initial conditions $I_o, \overline{I}_o \in \R$.
  Denote by $I$, resp.~$\overline{I}$,
  the solution to the Cauchy problem~\eqref{eq:VaxODE}
  with data $(S, I_o)$, resp.~$(\overline{S}, \overline{I}_o)$.
  Then for all $t > 0$,
  \begin{equation}
    \label{eq:WellPosedness4}
    \abs{I(t) - \overline{I}(t)}
    \leq \e^{K_tt}
    \left(\abs{I_o - \overline I_o} + \Lip(\alpha)
      \abs{\overline I_o} \int_0^t
      \e^{(\overline K_\tau - K_\tau) \, \tau }
    \norm{S(\tau) - \overline{S}(\tau)}_{\LL1((0, 1); \R)} \d{\tau} \right)
  \end{equation}
  where $K_t \coloneqq \beta + \Lip(\alpha)
  \norm{S}_{\LL\infty\left((0, t); \LL1((0,1);\R)\right)}$ and $\overline{K}_t \coloneqq \beta + \Lip(\alpha)
  \norm{\overline S}_{\LL\infty\left((0, t); \LL1((0,1);\R)\right)}$.
\end{lemma}




\begin{proof}
  We have
  \begin{displaymath}
    I(t) = I_o + \int_0^t\left(-\beta \, I(\tau) + \alpha(I(\tau))
      \int_0^1 S(\tau, \xi) \dd\xi\right)\dd\tau
  \end{displaymath}
  and similarly for $\overline {I}$.
  Thus
  \begin{align*}
    \abs{I (t) - \overline I(t)} 
    & \le 
    \abs{I_o - \overline I_o}
      + \beta \int_0^t  \abs{I (\tau) - \overline I(\tau)} \dd\tau
    \\
    & \quad
    +\int_0^t \abs{\alpha(I(\tau))\int_0^1 S(\tau, \xi) \dd\xi
    -
    \alpha(\overline I(\tau))\int_0^1 \overline S(\tau, \xi) \dd\xi
    }\dd\tau.
  \end{align*}
  Now we can add and subtract $\alpha(I(\tau))\int_0^1 \overline S(\tau, \xi) \dd\xi$ or $\alpha(\overline I(\tau))\int_0^1 S(\tau, \xi) \dd\xi$. In the first case we get
  \begin{align*}
    \abs{I (t) - \overline I(t)} 
    & \le 
    \abs{I_o - \overline I_o}
    + \beta \int_0^t  \abs{I (\tau) - \overline I(\tau)} \dd\tau
    +\int_0^t \abs{\alpha(I(\tau))} \cdot 
    \norm{ S(\tau) - \overline S(\tau)}_{\LL1((0,1);\R)} \dd\tau
    \\
    &+ \int_0^t
    \abs{\alpha(I(\tau)) - \alpha(\overline I(\tau))} \int_0^1 \abs{\overline S(\tau, \xi)} \dd\xi
    \dd\tau
    \\ 
    & \le \abs{I_o - \overline I_o}
    + \beta \int_0^t  \abs{I (\tau) - \overline I(\tau)} \dd\tau
    + \Lip(\alpha) \abs{I_o}  \int_0^t \e^{K_\tau \, \tau } \norm{ S(\tau) - \overline S(\tau)}_{\LL1((0,1);\R)} \dd\tau
    \\
    & + \Lip(\alpha) \norm{\overline S}_{\LL\infty((0,t);\LL1((0,1); \R))} \int_0^t \abs{I(\tau) - \overline I(\tau)} \dd\tau,
  \end{align*}
  where we used also~\eqref{eq:WellPosedness3} with
  $K_t \coloneqq \beta + \Lip(\alpha)
  \norm{S}_{\LL\infty\left((0, t); \LL1((0,1);\R)\right)}$.
  Setting also $\overline K_t \coloneqq \beta + \Lip(\alpha)
  \norm{\overline S}_{\LL\infty\left((0, t); \LL1((0,1);\R)\right)}$,
  by Gronwall Lemma we get
  \begin{displaymath}
    \abs{I (t) - \overline I(t)} 
    \le \e^{\overline K_t \, t} \left(
      \abs{I_o - \overline I_o}
      + \Lip(\alpha) \abs{I_o}  \int_0^t \e^{(K_\tau - \overline K_\tau)  \tau }
      \norm{ S(\tau) - \overline S(\tau)}_{\LL1((0,1);\R)} \dd\tau
    \right).
  \end{displaymath}
  In the second case, analogous computations lead to
  \begin{displaymath}
    \abs{I (t) - \overline I(t)} 
    \le  \e^{K_t \, t} \left(
      \abs{I_o - \overline I_o}
      + \Lip(\alpha) \abs{\overline I_o}
      \int_0^t  \e^{(\overline K_\tau - K_\tau) \, \tau }\norm{ S(\tau) - \overline S(\tau)}_{\LL1((0,1);\R)} \dd\tau
    \right).
  \end{displaymath}
  Therefore we deduce~\eqref{eq:WellPosedness4}.
\end{proof}

\begin{remark}
  Clearly, in \Cref{lmm:WellPosedness4}, the roles of $S$ and $\overline S$
  are symmetric. Therefore \eqref{eq:WellPosedness4} can be changed
  into
  \begin{equation*}
    |I(t) - \overline{I}(t)|
    \leq \e^{\overline K_tt}
    \left(\abs{I_o - \overline I_o} + \Lip(\alpha)
      \abs{\overline I_o} \int_0^t
      \e^{( K_\tau - \overline K_\tau) \, \tau }
    \norm{S(\tau) - \overline{S}(\tau)}_{\LL1((0, 1); \R)} \d{\tau} \right).
  \end{equation*}
\end{remark}

\subsubsection{The Mixed System}

This part contains the proofs of \Cref{th:WellPosedness} and
of~\Cref{th:Stability}.

\begin{proof}[\textbf{Proof of \Cref{th:WellPosedness}}]
  We divide the proof into various steps.
  \smallskip

  \textbf{Step 1: Local in time existence}. 
  Choose $T>0$ such that
  \begin{equation}
    \label{eq:constant-K-fixed-point}
    \left(\Lip(\alpha)\right)^2 \left(\abs{I_o} +
      \norm{S_o}_{\LL1((0,1);\R)}\right) 
    \norm{S_o}_{\LL1((0,1);\R)} T^2 \,\e^{2 K\,T} \le 1, 
  \end{equation}
    where the constant $K$ is defined by
    \begin{equation*}
      K\coloneqq \beta + \Lip(\alpha) \norm{S_o}_{\LL1((0,1);\R)}.
    \end{equation*}
  Consider the mappings
  \begin{equation*}
    \fonction{\Lambda}{\Czero([0, T]; \R^+)}
    {\Czero([0, T]; {\LL{2}}\left((0,1); \R^+\right))}{I}{S,}
  \end{equation*}
  where $S$ is the unique weak solution to \eqref{eq:VaxPDE}
  with data $(S_o, I)$ according to~\Cref{lmm:WellPosedness1}, and
  \begin{equation*}
    \fonction{\Upsilon}{\Czero([0, T]; {\LL{2}}\left((0,1); \R^+\right))}
    {\Czero([0, T]; \R^+)}{S}{I,}
  \end{equation*}
  where $I$ is the unique solution to~\eqref{eq:VaxODE} with data $(S, I_o)$,
  according to~\Cref{lmm:WellPosedness3}. The positivity of $\Lambda(I)$ and $\Upsilon(S)$ are ensured by~\Cref{lmm:WellPosedness1} and~\Cref{lmm:WellPosedness3} respectively.
 
  \smallskip  

  Fix $I, \overline{I} \in \Czero([0, T]; \R^+)$.
  By~\eqref{eq:WellPosedness4},
  defining $K_T \coloneqq \beta + \Lip(\alpha) \norm{\Lambda(I)}_{\LL\infty\left((0,T);
      \LL1\left((0,1); \R\right)\right)}$ and
  $\overline K_T \coloneqq \beta + \Lip(\alpha) \norm{\Lambda(\overline I)}
  _{\LL\infty\left((0,T);
      \LL1\left((0,1); \R\right)\right)}$,
  we have that, for every $t \in [0, T]$,
  \begin{equation}
    \label{eq:4}
    \begin{split}
      \abs{\left(\Upsilon \circ \Lambda(I)\right) (t) - \left(\Upsilon
          \circ \Lambda(\overline{I})\right) (t) } & \le \e^{K_T \, t}
      \, \Lip(\alpha) \abs{I_o} \int_0^t \e^{\overline K_T \tau}
      \norm{\Lambda(I)(\tau) -
        \Lambda(\overline I)(\tau)} _{\LL1\left((0,1);
          \R\right)}\dd\tau
      \\
      & \le \e^{2 K \, t} \, \Lip(\alpha) \abs{I_o} \int_0^t
      \norm{\Lambda(I)(\tau) - \Lambda(\overline I)(\tau)}
      _{\LL1\left((0,1); \R\right)}\dd\tau,
    \end{split}
  \end{equation}
  where we used the fact that $K_T \le K$ and
  $\overline K_T \le K$ by~\eqref{eq:WellPosedness-L1}.
  Using now~\eqref{eq:WellPosedness2a} and~\eqref{eq:constant-K-fixed-point},
  by~\eqref{eq:4} we deduce that, for all $t \in [0, T]$,
  \begin{align*}
    \abs{\left(\Upsilon \circ \Lambda(I)\right) (t)
    - \left(\Upsilon \circ \Lambda(\overline{I})\right) (t) }
    & \le \e^{2 K \, t} \, \left(\Lip(\alpha)\right)^2 \abs{I_o} \cdot 
      \norm{S_o}_{\LL1((0,1);\R)} \int_0^t
      \int_0^\tau \abs{I(\sigma) - \overline I (\sigma)}\dd\sigma \dd\tau
    \\ 
    & \le \frac{\left(\Lip(\alpha)\right)^2 \abs{I_o} \cdot 
      \norm{S_o}_{\LL1((0,1);\R)} T^2\, \e^{2 K\,T}}{2}
      \norm{I - \overline{I}}_{\Czero([0, T];\R)}
    \\
    & \le \frac{1}{2} \norm{I - \overline{I}}_{\Czero([0, T];\R)}.
  \end{align*}  
  Therefore
  \begin{displaymath}
    \norm{\Upsilon \circ \Lambda(I) - \Upsilon \circ \Lambda(\overline{I}) }
    _{\Czero([0,T]; \R)}
    \le\frac12 \norm{I-\overline{I}}_{\Czero([0,T];\R)},
  \end{displaymath}
  ensuring that $\Upsilon \circ \Lambda$ admits a unique fixed point
  $I^\star \in \Czero([0, T]; \R)$.  By construction
  and~\Cref{lmm:WellPosedness1}, the couple $(S^\star, I^\star)$,
  where $S^*\coloneqq \Lambda(I^*) \in \Czero([0, T]; \LL{2}((0,1);\R))$, is
  a solution to~\eqref{eq:VaxSystem} on $[0, T]$ with initial data
  $(S_o, I_o)$, according to \Cref{def:solution}.
   
  \smallskip

  \textbf{Step 2: A priori estimate.}
  Let $(S, I)$ be a solution to \eqref{eq:VaxSystem}
  on the time interval $[0, T]$ for some $T>0$,
  according to~\Cref{def:solution}.

  Clearly, by \Cref{lmm:WellPosedness1}, the
  bounds~\eqref{eq:WellPosedness} for $S$ hold for every $t \in [0, T]$.

  We now prove the bound~\eqref{eq:estimate-sol_I} for $t \in [0, T]$,
  related to $I$.  Let $(S_o^{(n)})_n$ be a
  sequence of nonnegative $\Ck{1}$ functions that converges to $S_o$
  in $\LL{1}\left((0,1); \R\right)$.
  For all $n \in \N$, call $S_n$ the distributional solution to
  \begin{equation*}
    \left\{
      \begin{aligned}
        \p_t S_n + \p_\xi (g(\xi, u_1(t), u_2(t)) S_n)
        & = - (f + \alpha (I)) S_n \\
        S_n(0, \xi) & = S_o^{(n)}(\xi),
      \end{aligned}
    \right.
  \end{equation*}
  which is also a classical solution since $S_o^{(n)}$ is smooth.
  Let $\varphi \in \Cc\infty\left((0, +\infty); \R^+\right)$.
  For all
  $n \in \N$ and $t \in [0, T]$, we have
  \begin{align*}
    & \quad
      \int_0^{+\infty} \varphi'(t) \left(\int_0^1 S_n(t, \xi) \dd \xi + I(t)\right)
      \dd t
      = \int_0^{+\infty} \varphi'(t)
      \int_0^1 S_n(t, \xi) \dd \xi \dd t
      + \int_0^{+\infty} \varphi'(t) I(t) \dd t
    \\
    & = \int_0^{+\infty} \int_0^1 \varphi'(t) S_n(t, \xi)
      \dd \xi \dd t + \beta \int_0^{+\infty} \varphi(t) I(t) \dd t
      - \int_0^{+\infty} \varphi(t) \alpha (I(t)) \int_0^1 S(t, \xi) \dd \xi \dd t
    \\
    & = \int_0^{+\infty} \!\!\int_0^1 
      {\partial_\xi \left(S_n(t, \xi) g(\xi, u_1(t),
      u_2(t))\right)} \varphi(t) \dd \xi \dd t
      + \!\! \int_0^{+\infty} \!\!\int_0^1 S_n(t, \xi) f(\xi)
      \varphi(t) \dd \xi \dd t
    \\
    & \quad + \int_0^{+\infty} \int_0^1 S_n(t, \xi) \alpha(I(t))
      \varphi(t) \dd \xi \dd t
      - \int_0^1 S_o^{(n)}(\xi) \varphi(0) \dd \xi \dd t
    \\
    & \quad
      + \beta \int_0^{+\infty} \varphi(t) I(t) \dd t
      - \int_0^{+\infty} \varphi(t) \alpha (I(t)) \int_0^1 S(t, \xi) \dd \xi \dd t
    \\
    & = \int_0^{+\infty} \!\!\int_0^1 S_n(t, \xi) f(\xi)
      \varphi(t) \dd \xi \dd t
      + \int_0^{+\infty} \int_0^1 S_n(t, \xi) \alpha(I(t))
      \varphi(t) \dd \xi \dd t
    \\
    & \quad
      + \beta \int_0^{+\infty} \varphi(t) I(t) \dd t
      - \int_0^{+\infty} \varphi(t) \alpha (I(t)) \int_0^1 S(t, \xi) \dd \xi \dd t
    \\
    & \ge \int_0^{+\infty} \int_0^1 S_n(t, \xi) \alpha(I(t))
      \varphi(t) \dd \xi \dd t
      - \int_0^{+\infty} \varphi(t) \alpha (I(t)) \int_0^1 S(t, \xi) \dd \xi \dd t.
  \end{align*}
  Using \Cref{lmm:WellPosedness2}, the bound~\eqref{eq:WellPosedness},
  and passing to the limit as $n\to+\infty$ in the previous estimates, we
  deduce that
  \begin{equation*}
    \int_0^{+\infty} \varphi'(t) \left(\int_0^1 S(t, \xi) \dd \xi + I(t)\right)
      \dd t \ge 0,
  \end{equation*}
  so that the distributional derivative of
  \begin{equation*}
    t \mapsto \int_0^1 S(t, \xi) \dd \xi + I(t)
  \end{equation*}
  is a positive measure.
  Thus, for every $t \in [0, T]$, we obtain that
  \begin{equation*}
    I(t) \le 
    \int_0^1 S(t, \xi) \d{\xi} + I(t) \leq \norm{S_o}_{\LL1\left((0,1); \R\right)}
    + I_o,
  \end{equation*}
  proving~\eqref{eq:estimate-sol_I} for $t \in [0, T]$.

  \textbf{Step 3: Global in time existence.} Introduce the time
  \begin{equation*}
    T^* \coloneqq  \sup \biggl\{ \tau > 0 \; : \;
    \eqref{eq:VaxSystem} \;
    \text{admits a solution on} \; [0, t] \; \text{for all} \; t \in [0, \tau[
    \biggr\}.
  \end{equation*}

  Step 1 ensures that $T^*$ is well-defined and $T^* > 0$. Suppose,
  by contradiction, that
  $T^* < +\infty$ and choose $0 < \tau < \frac{T^*}{2}$ such that
  \begin{equation}
    \label{eq:tau}
    2 \left(\Lip(\alpha)\right)^2 \left(\abs{I_o}
      + \norm{S_o}_{\LL1((0,1);\R)} \right)
    \norm{S_o}_{\LL1((0,1);\R)} \tau^2 \,\e^{2 \, K\,\tau} \le 1. 
  \end{equation}

  Set $T \coloneqq T^* - \tau/2$. Since $T < T^*$, system~\eqref{eq:VaxSystem}
  admits a solution on $[0, T]$, say $(S_1,
  I_1)$. 
  By~\eqref{eq:WellPosedness}, \eqref{eq:estimate-sol_I}, and~\eqref{eq:tau}
  we deduce that
  \begin{align*}
    & \quad \left(\Lip(\alpha)\right)^2 \left(\abs{I_1(T)} +
      \norm{S_1(T)}_{\LL1((0,1);\R)}\right) 
      \norm{S_1(T)}_{\LL1((0,1);\R)} \tau^2 \,\e^{2 K\,\tau}
    \\
    & \le \left(\Lip(\alpha)\right)^2 \left(\abs{I_o} +
      2 \norm{S_o}_{\LL1((0,1);\R)}\right) 
      \norm{S_o}_{\LL1((0,1);\R)} \tau^2 \,\e^{2 K\,\tau}
    \\
    & \le 2 \left(\Lip(\alpha)\right)^2 \left(\abs{I_o} +
      \norm{S_o}_{\LL1((0,1);\R)}\right) 
      \norm{S_o}_{\LL1((0,1);\R)} \tau^2 \,\e^{2 K\,\tau}
    \\
    & \le 1.
  \end{align*}
  Therefore, by Step 1, we can construct a solution $(S_2, I_2)$
  to~\eqref{eq:VaxSystem} on $[T, T+\tau]$ with initial condition
  $(S_1(T), I_1(T))$.
  Clearly the concatenation $(S, I)$, patching together $(S_1, I_1)$
  and $(S_2, T_2)$, is a solution to~\eqref{eq:VaxSystem} on $[0, T+\tau]$,
  with $T+\tau = T^* + \frac{\tau}{2} > T^*$,
  contradicting the choice of $T^*$. Therefore,
  $T^* = +\infty$, concluding the proof of the global in time existence.
  \smallskip

  \textbf{Step 4: Uniqueness of solution.}
  Assume, by contradiction, that there are two distinct solutions,
  namely $(S, I)$ and $(\tilde S, \tilde I)$, to~\eqref{eq:VaxSystem}
  with the same initial condition~\eqref{eq:IC}.
  For every $t > 0$ and $\xi \in (0,1)$,
  a similar reasoning as in \cite[Theorem~3.2]{Keimer2017}
  implies that
  \begin{align}
    \label{eq:5}
    S(t, \xi)
    & = S_o(X(0; t, \xi)) \exp\left(-\int_0^t \left[f(X(s; t, \xi))
      + \alpha(I\left(s\right))
      + \p_\xi G(X(s; t, \xi), s) \right]\dd s\right) 
    \\
    \label{eq:6}
    \tilde S(t, \xi)
    & = S_o(X(0; t, \xi)) \exp\left(-\int_0^t \left[f(X(s; t, \xi))
      + \alpha(\tilde I\left(s\right))
      + \p_\xi G(X(s; t, \xi), s) \right]\dd s\right)
    \\
    \label{eq:7}
    I(t)
    & = I_o - \int_0^t \beta I(\tau) \dd \tau + \int_0^t \alpha(I(\tau))
      \left(\int_0^1 S(\tau, \xi) \dd \xi\right) \dd \tau
    \\
    \label{eq:8}
    \tilde I(t)
    & = I_o - \int_0^t \beta \tilde I(\tau) \dd \tau
      + \int_0^t \alpha(\tilde I(\tau))
      \left(\int_0^1 \tilde S(\tau, \xi) \dd \xi\right) \dd \tau,
  \end{align}
  where $G$ is defined in~\eqref{eq:G-flux}.

  Using~\eqref{eq:5}, \eqref{eq:6}, the change of variable $y = X(0; t, \xi)$, 
  \ref{hyp:(alpha)}, and~\ref{hyp:(F)}, we deduce that, for every
  $t > 0$,
  \begin{align}
    \nonumber
    \norm{S(t) - \tilde S(t)}_{\LL1\left((0,1); \R\right)}
    & \le \int_0^1 \abs{S_o(X(0; t, \xi))}
      \exp\left(-\int_0^t \left[f(X(s; t, \xi))
      + \p_\xi G(X(s; t, \xi), s) \right]\dd s\right)
    \\
    \nonumber
    & \quad \quad \times
      \abs{\exp\left(-\int_0^t \alpha(I\left(s\right))
      \dd s\right) - \exp\left(-\int_0^t \alpha(\tilde I\left(s\right))
      \dd s\right)} \dd \xi
    \\
    \nonumber
    & \le \int_0^1 \abs{S_o(y)}
      \exp\left(-\int_0^t f(y)
      \dd s\right) \dd y
      \int_0^t
      \abs{\alpha(I\left(s\right)) - \alpha(\tilde I\left(s\right))} \dd s
    \\
    \label{eq:L1-dist-S-S}
    & \le \Lip(\alpha) \norm{S_o}_{\LL1\left((0,1); \R\right)}
      \norm{I - \tilde I}_{\LL1\left((0,t); \R\right)}.
  \end{align}
  Using~\eqref{eq:7}, \eqref{eq:8}, 
  \ref{hyp:(alpha)}, \eqref{eq:WellPosedness}, \eqref{eq:estimate-sol_I},
  and~\eqref{eq:L1-dist-S-S},  we deduce that, for every
  $t > 0$,
  \begin{align*}
    \abs{I(t) - \tilde I(t)}
    & \le \beta \int_0^t \abs{I(\tau) - \tilde I(\tau)} \dd \tau
      + \int_0^t \alpha(I(\tau)) \abs{\int_0^1 S(\tau, \xi) \dd \xi
      - \int_0^1 \tilde S(\tau, \xi) \dd \xi} \dd \tau
    \\
    & \quad + \int_0^t \abs{\int_0^1 \tilde S(\tau, \xi) \dd \xi} \cdot 
      \abs{\alpha(I(\tau)) - \alpha(\tilde I(\tau))} \dd \tau
    \\
    & \le \beta \int_0^t \abs{I(\tau) - \tilde I(\tau)} \dd \tau
      + \int_0^t \alpha(I(\tau)) \norm{S(\tau) - \tilde S(\tau)}
      _{\LL1\left((0,1);\R\right)}\dd \tau
    \\
    & \quad + \Lip(\alpha) \int_0^t \norm{\tilde S(\tau)}
      _{\LL1\left((0,1); \R\right)}
      \abs{I(\tau) - \tilde I(\tau)} \dd \tau
    \\
    & \le \beta \int_0^t \abs{I(\tau) - \tilde I(\tau)} \dd \tau
      + \left(\Lip(\alpha)\right)^2 \norm{S_o}
      _{\LL1\left((0,1);\R\right)}\int_0^t \abs{I(\tau)}
      \norm{I - \tilde I}_{\LL1\left((0,\tau); \R\right)}\dd \tau
    \\
    & \quad + \Lip(\alpha) \norm{\tilde S_o}_{\LL1\left((0,1);\R\right)} \int_0^t 
      \abs{I(\tau) - \tilde I(\tau)} \dd \tau
    \\
    & \le \left(\beta + \left(\Lip(\alpha)\right)^2 \norm{S_o}
      _{\LL1\left((0,1);\R\right)}
      \left(I_o + \norm{S_o}_{\LL1\left((0,1); \R\right)}\right) t
      + \Lip(\alpha) \norm{\tilde S_o}_{\LL1\left((0,1);\R\right)}\right)
    \\
    & \quad \quad \times
      \int_0^t \abs{I(\tau) - \tilde I(\tau)} \dd \tau.      
  \end{align*}
  Gronwall Lemma implies that $I(t) = \tilde I(t)$ for every $t >0$.
  Finally~\eqref{eq:L1-dist-S-S} implies that $S(t) = \tilde S(t)$ in
  $\LL1\left((0,1); \R\right)$ for all $t > 0$. 
  
  In the case $S_o \in \HH1\left((0,1); \R^+\right)$, the required higher regularity of the solution follows by the representation formula~\eqref{eq:charact-curves}.
  This concludes the proof.  
\end{proof}

\begin{proof}[\textbf{Proof of~\Cref{th:Stability}}]
  By~\Cref{lmm:WellPosedness2}, for all $t > 0$ we have
  \begin{align*}
    \norm{S(t)- \overline S(t)}_{\LL1((0,1);\R)} 
    & \le \norm{S_o - \overline S_o}_{\LL1((0,1); \R)}
    + \Lip(\alpha) \norm{\overline S_o}_{\LL1((0,1);\R)}  
    \norm{I - \overline I}_{\LL1((0,t);\R)}.
  \end{align*}
  On the other hand, by~\Cref{lmm:WellPosedness4}, for all $t > 0$ we have
  \begin{align*}
    \abs{I(t) - \overline I(t)}
    & \le \e^{K_1 \, t} \left(
      \abs{I_o-\overline I_o} + \Lip(\alpha) \abs{\overline I_o}\int_0^t 
      \e^{K_2 \, \tau}
      \norm{S(\tau)- \overline S(\tau)}_{\LL1((0,1);\R)}  \dd\tau
    \right),
  \end{align*}
  where $K_1$ and $K_2$ are defined in~\eqref{eq:K12}.
  Hence, for all $t > 0$
  \begin{align*}
    & \norm{S(t)- \overline S(t)}_{\LL1((0,1);\R)}
    + \abs{I(t) - \overline I(t)}
    \\
    & \le \norm{S_o - \overline S_o}_{\LL1((0,1); \R)}
    +  \e^{K_1 \, t} \abs{I_o-\overline I_o}
    \\
    & \quad +\Lip(\alpha) 
    \left(\norm{\overline S_o}_{\LL1((0,1);\R)} + \e^{K_1 \, t }\abs{\overline I_o} \right)
      \int_0^t \e^{K_2 \, \tau}\left(
        \norm{S(\tau) - \overline S(\tau)}_{\LL1((0,1);\R)}
        + \abs{I(\tau) - \overline I(\tau)}
    \right)\dd\tau.
  \end{align*}
  Set $\mathcal K_t$ as in~\eqref{eq:mathcalK}.
  An application of Gronwall Lemma yields
  \begin{align*}
    & \quad \norm{S(t)- \overline S(t)}_{\LL1((0,1);\R)}
    + \abs{I(t) - \overline I(t)}
    \\
    & \le \norm{S_o - \overline S_o}_{\LL1((0,1); \R)}
    +  \e^{K_1 \, t} \abs{I_o-\overline I_o}
    \\
    & \quad + \mathcal{K}_t
    \int_0^t \left(
      \norm{S_o - \overline S_o}_{\LL1((0,1); \R)}
    +  \e^{K_1 \, \tau} \abs{I_o-\overline I_o} 
    \right)
    \e^{K_2 \, \tau}
    \exp\left(\int_\tau^t \e^{K_2 \, s} \mathcal{K}_s \dd{s}
    \right)\dd\tau
    \\
    & \le \left[\norm{S_o - \overline S_o}_{\LL1((0,1); \R)}
      +  \e^{K_1 \, t} \abs{I_o-\overline I_o}\right]
      \left(1 + t \e^{K_2 t} \mathcal K_t \exp(t \e^{K_2 t} \mathcal K_t)\right).
  \end{align*}
 concluding the proof.  
\end{proof}

\subsection{Decoupled optimization}
\label{ssec:DecoupledOptimization}

\begin{proof}[\textbf{Proof of~\Cref{th:DecoupledOptimization}}]
  Observe first that, since~\ref{hyp:(G)} holds, for all
  $u_1 \in \cU_1$ and
  $u_2 \in \cU_2$, the functional $\mathcal F$,
  defined in~\eqref{eq:CostFunctional}, satisfies the estimate
  \begin{equation}
    \label{eq:bound_cF}
    \abs{\cF(S_o, I_o, u_1, u_2)}
    \leq \frac{\kappa (\norm{S_o}_{\LL{1}((0,1); \R)} + I_o + M_1) + M_2 + \gamma}{\theta},
  \end{equation}
  where $\gamma$ is defined in~\ref{hyp:(G)}.


  Consider~\Cref{it:u1}. 
  Fix $u_2 \in \cU_2$ and consider a minimizing
  sequence of controls $(u_1^{(n)})_n$, taking values in
  $\cU_1$, such that
  \begin{equation}
    \label{eq:minimizing-sequence}
    \lim_{n \to +\infty}\cF(S_o, I_o, u_1^{(n)}, u_2)
    = \inf_{u_1 \in \cU_1} \cF (S_o, I_o, u_1, u_2).
  \end{equation}
  The sequence $(u_1^{(n)})_n$ is bounded in $\LL{\infty}$, therefore
  it admits a subsequence, which we do not relabel, that converges in
  the weak* sense: there exists
  $\overline{u}_1 \in \cU_1$ such that, for
  every $\psi \in \LL{1}(\R^+; \R)$,
  \begin{equation*}
    \lim_{n \to + \infty}
    \int_{0}^{+\infty} u_1^{(n)}(t) \psi(t) \dd{t} 
    = \int_{0}^{+\infty} \overline{u}_1(t) \psi(t) \dd{t}.
  \end{equation*}
  For all $n \in \N$, call $(S_n, I_n)$ the solution
  to~\eqref{eq:VaxSystem} with controls $(u_1^{(n)}, u_2)$ and initial
  data $(S_o, I_o)$.  By \Cref{rmk:characteristics-representation}, we
  have that
  \begin{equation}
    \label{eq:S_n}
    S_n(t, \xi) = S_o (X_n(0; t, \xi)) \exp \left(-\int_{0}^{t}
      \left[c_n(s,X_n(s; t, \xi)) + \p_1 g(X_n(s; t, \xi), u_1^{(n)}(s), u_2(s))
      \right] \dd{s} \right),
  \end{equation}
  where $c_n (s,\zeta) = f(\zeta) + \alpha (I_n(s))$ and
  $X_n = X_n(s; t, \xi)$ solves~\eqref{eq:charact-curves} with $u_1$
  replaced by $u_1^{(n)}$.  Hence, by~\ref{hyp:(G-1)}, for every
  $s \in \R^+$
  \begin{equation}
    \label{eq:charact-n}
    \begin{split}
      X_n(s; t, \xi) & = \xi + \int_t^s g\left(X_n\left(\tau; t,
          \xi\right), u_1^{(n)}(\tau), u_2(\tau)\right) \dd \tau
      \\
      & = \xi + \int_t^s g_1\left(X_n\left(\tau; t, \xi\right) \right)
      \dd \tau + \int_t^s g_2\left(X_n\left(\tau; t, \xi\right)\right)
      u_1^{(n)}(\tau) \dd \tau + \int_t^s g_3\left(X_n\left(\tau; t,
          \xi\right)\right) u_2(\tau) \dd \tau.
    \end{split}
  \end{equation}
  For all $(t, \xi) \in (0,\infty) \times [0, 1]$,
  the sequence $s \mapsto X_n(s; t, \xi)$ is uniformly equicontinuous
  and bounded on any compact subset of $\R^+$, since
  $\abs{\p_1 X_n(s; t, \xi)}\leq \gamma$, by~\ref{hyp:(G)}.  Then,
  Ascoli-Arzelà Theorem implies that, up to a subsequence, it
  converges to a function $s \mapsto \Psi_{t, \xi}(s)$.

  Using~\ref{hyp:(G-1)}, the Lebesgue Theorem, and the weak* convergence of
  $u_1^{(n)}$, we deduce that, for all $s \in \R^+$,
  \begin{equation}
    \label{eq:limits-g1-g3}
    \begin{split}
      \lim_{n \to + \infty} \int_t^s g_1\left(X_n\left(\tau; t,
          \xi\right) \right) \dd \tau & = \int_t^s g_1\left(\Psi_{t,
          \xi} (\tau) \right) \dd \tau
      \\
      \lim_{n \to + \infty} \int_t^s g_3\left(X_n\left(\tau; t,
          \xi\right) \right) u_2(\tau) \dd \tau & = \int_t^s
      g_3\left(\Psi_{t, \xi} (\tau) \right) u_2(\tau) \dd \tau
      \\
      \lim_{n \to + \infty}
      \int_t^s g_2\left(\Psi_{t, \xi}\left(\tau\right)\right) u_1^{(n)}(\tau)
      \dd \tau
      & = \int_t^s g_2\left(\Psi_{t, \xi}\left(\tau\right)\right)
      \overline{u}_1(\tau) \dd \tau.
    \end{split}
  \end{equation}
  Note that
  \begin{align*}
    \abs{\int_t^s \left(g_2\left(X_n\left(\tau; t, \xi\right)\right)
    - g_2\left(\Psi_{t, \xi}(\tau)\right)\right) u_1^{(n)}(\tau)
    \dd \tau}
    & \le
      \int_{\min\{t, s\}}^{\max\{t, s\}}
      \abs{g_2\left(X_n\left(\tau; t, \xi\right)\right)
      - g_2\left(\Psi_{t, \xi}(\tau)\right)} \abs{u_1^{(n)}(\tau)} \dd \tau
    \\
    & \le
      M_1 \int_{\min\{t, s\}}^{\max\{t, s\}}
      \abs{g_2\left(X_n\left(\tau; t, \xi\right)\right)
      - g_2\left(\Psi_{t, \xi}(\tau)\right)} \dd \tau
    \\
    & \le
      M_1 \sup_{\tau} 
      \abs{g_2\left(X_n\left(\tau; t, \xi\right)\right)
      - g_2\left(\Psi_{t, \xi}(\tau)\right)} \abs{t-s}
  \end{align*}
  and so
  \begin{equation}
    \label{eq:9}
    \lim_{n \to + \infty}
    \int_t^s \left(g_2\left(X_n\left(\tau; t, \xi\right)\right)
      - g_2\left(\Psi_{t, \xi}(\tau)\right)\right) u_1^{(n)}(\tau)
    \dd \tau = 0.
  \end{equation}
  Using~\eqref{eq:limits-g1-g3} and~\eqref{eq:9} and passing to
  the limit as $n \to + \infty$ in~\eqref{eq:charact-n}, we deduce that
  \begin{equation}
    \label{eq:charact-psi}
    \begin{split}
      \Psi_{t, \xi}(s)
      & = 
      \xi + \int_t^s g_1\left(\Psi_{t, \xi} \left(\tau\right) \right)
      \dd \tau + \int_t^s g_2\left(\Psi_{t, \xi} \left(\tau\right)\right)
      \overline{u}_1(\tau) \dd \tau + \int_t^s g_3
      \left(\Psi_{t, \xi}\left(\tau\right)\right) u_2(\tau) \dd \tau
      \\
      & = \xi + \int_t^s g\left(\Psi_{t, \xi}(\tau), \overline{u}_1(\tau),
        u_2(\tau)\right) \dd \tau
    \end{split}
  \end{equation}
  for every $s \in \R^+$. Hence $\Psi_{t, \xi}$ solves
  the Cauchy problem~\eqref{eq:charact-curves} with the controls
  $\bar u_1$ and $u_2$.
  \smallskip

  The sequence $t \mapsto I_n(t)$ is uniformly equicontinuous and
  bounded on any compact subset $\R^+$. Indeed, the
  boundedness follows directly from~\eqref{eq:estimate-sol_I}, while
  by~\ref{hyp:(alpha)}, \eqref{eq:WellPosedness}
  and~\eqref{eq:estimate-sol_I}, for all $t \in \R^+$
  \begin{displaymath}
    \abs{\dot I_n(t)}
    \le 
    \left(\beta + \Lip(\alpha) \norm{S_o}_{\LL1((0,1);\R)}\right)
     \left(\norm{S_o}_{\LL1((0,1);\R)} + I_o\right).
  \end{displaymath} 
  Combining Ascoli-Arzelà Theorem
  and a standard diagonal procedure, we can find
  $\overline I \in \Czero(\R^+; \R^+)$ and a subsequence of
  $(I_n)_n$ that converges uniformly to $\overline I$ on any compact subset
  of $\R^+$.  
  
  Define $\overline S$, for $t\in \R^+$ and $\xi \in (0,1)$, as
  \begin{equation}
    \label{eq:bar-S-control}
    \overline{S}(t,\xi) \coloneqq 
    S_o (\Psi_{t,\xi}(0)) \exp \left(-\int_{0}^{t}
      \left[f(\Psi_{t,\xi}(s)) + \alpha(\overline{I}\left(s\right))
        + \p_\xi g(\Psi_{t,\xi}(s), \overline{u}_1(s), u_2(s)) \right]\dd{s}
    \right),
  \end{equation}
  that is, $\overline S$ is a solution to the partial differential
  equation~\eqref{eq:VaxPDE} with data $(S_o, \overline{I})$
  and controls $(\overline{u}_1, u_2)$.
  We claim that, for every $t > 0$,
  \begin{equation}
    \label{eq:10}
    \lim_{n \to +\infty} \norm{S_n(t) - \overline S(t)}_{\LL2\left((0,1); \R\right)}
    = 0.
  \end{equation}
  Fix $t>0$, $\eps > 0$, and choose
  $S_{o, \eps} \in \Ck\infty\left((0,1); \R^+\right)$
  such that $\norm{S_o - S_{o, \eps}}_{\LL2\left((0,1); \R\right)} < \eps$.
  Therefore
  \begin{align*}
    \left(\int_0^1 \abs{S_o(\Psi_{t,\xi}(0)) - S_o(X_n\left(0; t, \xi\right))}^2
    \dd \xi\right)^\frac12
    & \le \left(\int_0^1 \abs{S_o(\Psi_{t,\xi}(0)) - S_{o, \eps}
      (\Psi_{t,\xi}(0))}^2\dd\xi\right)^\frac12
    \\
    & \quad + \left(\int_0^1 \abs{S_{o, \eps}(\Psi_{t,\xi}(0))
      - S_{o, \eps}(X_n\left(0; t, \xi\right))}^2\dd\xi\right)^\frac12
    \\
    & \quad + \left(\int_0^1 \abs{S_{o, \eps}(X_n\left(0; t, \xi\right)) - S_{o}
      (X_n\left(0; t, \xi\right))}^2\dd\xi\right)^\frac12.
  \end{align*}
  Using the change of variable $y = \Psi_{t, \xi}(0)$
  and~\ref{hyp:(G)}, we deduce that
  \begin{align*}
    \int_0^1 \abs{S_o(\Psi_{t,\xi}(0)) - S_{o, \eps}(\Psi_{t,\xi}(0))}^2\dd\xi
    & = \int_0^1 \abs{S_o(y) - S_{o, \eps}(y)}^2
      \e^{\int_0^t \partial_1 g\left(\Psi_{0,y}(s), \overline u_1(s), u_2(s)
      \right) \dd s} \dd y
    \\
    & \le \e^{\gamma \, t} \int_0^1 \abs{S_o(y) - S_{o, \eps}(y)}^2 \dd y
      <  \e^{\gamma \, t} \eps^2.
  \end{align*}
  Similarly, we obtain that
  \begin{equation*}
    \int_0^1 \abs{S_{o, \eps}(X_n\left(0; t, \xi\right)) - S_{o}
      (X_n\left(0; t, \xi\right))}^2\dd\xi <  \e^{\gamma \, t} \eps^2 
  \end{equation*}
  and so
  \begin{align*}
    \left(\int_0^1 \abs{S_o(\Psi_{t,\xi}(0)) - S_o(X_n\left(0; t, \xi\right))}^2
    \dd \xi\right)^\frac12
    \le 2 \,\e^{\frac{\gamma \, t}{2}} \eps
    + \left(\int_0^1 \abs{S_{o, \eps}(\Psi_{t,\xi}(0))
    - S_{o, \eps}(X_n\left(0; t, \xi\right))}^2\dd\xi\right)^\frac12.
  \end{align*}
  Moreover, Lebesgue Theorem implies that
  \begin{equation*}
    \lim_{n \to + \infty} \int_0^1 \abs{S_{o, \eps}(
      \Psi_{t,\xi}(0)) - S_{o, \eps}(X_n\left(0; t, \xi\right))}^2\dd\xi = 0.
  \end{equation*}
  Hence we conclude that
  \begin{equation*}
    \lim_{n \to +\infty}
    \left(\int_0^1 \abs{S_o(\Psi_{t,\xi}(0)) - S_o(X_n\left(0; t, \xi\right))}^2
      \dd \xi\right)^\frac12
    \le 2 \e^{\frac{\gamma \, t}{2}} \eps.
  \end{equation*}
  Since the previous inequality holds for every $\eps > 0$, we conclude
  that
  \begin{equation}
    \label{eq:first-limit}
    \lim_{n \to +\infty}
    \left(\int_0^1 \abs{S_o(\Psi_{t,\xi}(0)) - S_o(X_n\left(0; t, \xi\right))}^2
      \dd \xi\right)^\frac12
    = 0.
  \end{equation}
  By~\ref{hyp:(alpha)}, \ref{hyp:(F)}, \ref{hyp:(G-1)},
  \eqref{eq:estimate-sol_I}, the uniform convergence of $X_n$ and of
  $I_n$ implies that the sequence of functions
  \begin{align*}
    \xi \mapsto
    \exp \left(-\int_{0}^{t}
    \left[f\left(X_n(s; t, \xi)\right) + \alpha\left(I_n(s)\right)
    + \p_1 g(X_n(s; t, \xi), u_1^{(n)}(s), u_2(s))\right]
    \dd{s} \right)
  \end{align*}
  converges in $\LL2\left((0,1); \R\right)$ to the function
  \begin{align*}
    \xi \mapsto
    \exp \left(-\int_{0}^{t}
    \left[f\left(\Psi_{t, \xi}(s)\right) + \alpha\left(\overline I(s)\right)
    + \p_1 g(\Psi_{t, \xi}(s), \overline u_1(s), u_2(s))\right]
    \dd{s} \right).
  \end{align*}
  Thus, using~\eqref{eq:first-limit}, we conclude that~\eqref{eq:10} holds,
  so that the claim is proved.

  Since 
  \begin{displaymath}
    I_n(t) = I_o + \int_0^t \left(-\beta I_n(\tau) 
    + \alpha(I_n(\tau)) \int_0^1 S_n(\tau, \xi)\dd\xi\right)\dd\tau,
  \end{displaymath}
  passing to the limit we obtain that
  \begin{displaymath}
    \overline{I}(t) = I_o + \int_0^t \left(-\beta \overline{I}(\tau) 
      + \alpha(\overline{I}(\tau)) \int_0^1 \overline{S}(\tau, \xi)\dd\xi\right)
    \dd\tau,
  \end{displaymath}
  that is, $\overline{I}$ is a solution to~\eqref{eq:VaxODE}
  with data $(\overline{S}, I_o)$.

  Therefore, the couple $(\overline{S}, \overline{I})$ is the 
  solution to~\eqref{eq:VaxSystem} with controls 
  $(\overline{u}_1, u_2)$ and initial data $(S_o, I_o)$.





  Finally, using~\ref{hyp:(G-1)}, \eqref{eq:minimizing-sequence},
  and~\eqref{eq:CostFunctional}, we deduce that
  \begin{align*}
    \inf_{u_1 \in \cU_1} \cF(S_o, I_o, u_1, u_2)
    & = \liminf_{n \to
      +\infty} \cF(S_o, I_o, u_1^{(n)}, u_2)
    \\
    & \ge \liminf_{n \to +\infty} \, \kappa \int_{0}^{+\infty}
      \e^{-\theta t} (I_n(t) + u_1^{(n)}(t)) \dd{t}
    \\
    & \quad - \lim_{n \to +\infty}
      \int_{0}^{+\infty} \e^{-\theta t} \biggl[ u_2(t)
      + \int_0^1 g(\xi, u_1^{(n)}(t), u_2(t)) \d{\xi} \biggr] \dd{t}
    \\
    & = \kappa \int_{0}^{+\infty}
      \e^{-\theta t} (\overline I(t) + \overline u_1(t)) \dd{t}
    \\
    & \quad - 
      \int_{0}^{+\infty} \e^{-\theta t} \biggl[ u_2(t)
      + \int_0^1 g(\xi, \overline u_1(t), u_2(t)) \d{\xi} \biggr] \dd{t}
    \\
    & = \cF(S_o, I_o, \overline u_1, u_2)
  \end{align*}
  ensuring that $\overline{u}_1$ is a minimizer of
  $u_1 \mapsto \cF\left(S_o, I_o, u_1, u_2\right)$ in
  $\cU_1$.  \smallskip

  The proof of~\Cref{it:u2} in~\Cref{th:DecoupledOptimization} is
  similar, thus omitted.
\end{proof}

\subsection{Two-player game}
\label{ssec:Game}

\subsubsection{Dynamic Programming Principle}
\label{sssec:DPP}


\begin{proof}[\textbf{Proof of \Cref{th:DPP}}]
    We only give the details for $V$ since the proof for $U$ is similar.

    Fix $T > 0$ and let us call $\tilde V(S_o, I_o)$ the right hand side
    of \eqref{eq:DPP}, i.e.
    \begin{equation*}
      \tilde V(S_o, I_o) \coloneqq
        \adjustlimits\inf_{\mathcal S_1 \in \Gamma} \sup_{u_2 \in \cU_2} 
        \left\{
          \displaystyle \int_{0}^{T} \e^{-\theta t} \ell\left(
           I_{\mathcal{S}_1,u_2}(t),\mathcal{S}_1(u_2)(t)
          , u_2(t)\right) \dd{t}
          + V\left(S_{\mathcal{S}_1,u_2}(T)), I_{\mathcal{S}_1,u_2}(T)\right) \e^{-\theta T}       
        \right\},
    \end{equation*}
    where the couple $(S_{\cS_1,u_2}, I_{\cS_1,u_2})$ denotes the solution
    to~\eqref{eq:VaxSystem} with initial datum $(S_o, I_o)$ and
    controls $\mathcal{S}_1(u_2)$ and $u_2$.

    \textbf{First we prove that} $V(S_o, I_o) \leq \tilde{V}(S_o, I_o)$. 
    
    Let $\eps > 0$ be fixed.
    Choose a strategy $\hat{\cS}_1 \in \Gamma$ such that 
    \begin{equation}
      \label{eq:proofDPP1a}
      \tilde{V}(S_o, I_o) + \eps \geq \sup_{u_2 \in \cU_2} 
      \left\{\int_0^T \e^{-\theta t} \ell
        \left( I_{\hat{\cS}_1,u_2}(t), \hat{\cS}_1(u_2)(t), u_2(t)\right) \dd{t} 
        + \e^{-\theta T} V\left(S_{\hat{\cS}_1,u_2}(T), I_{\hat{\cS}_1,u_2}(T)\right)\right\},
    \end{equation}
    where the couple $(S_{\hat{\cS}_1,u_2}, I_{\hat{\cS}_1,u_2})$ denotes the solution
    to~\eqref{eq:VaxSystem} with initial datum $(S_o, I_o)$ and
    controls $\mathcal{\hat S}_1(u_2)$ and $u_2$.


    Fix now $u_2 \in \cU_2$ arbitrarily.
    Using \eqref{eq:proofDPP1a}, we obtain that
    \begin{equation}
      \label{eq:11}
      \tilde{V}(S_o, I_o)
      \geq 
      \int_0^T \e^{-\theta t}
      \ell\left(
        I_{\hat{\cS}_1,u_2}(t), \hat{\cS}_1(u_2)(t), u_2(t)\right) \dd{t} 
      + \e^{-\theta T} 
      V\left(S_{\hat{\cS}_1,u_2}(T), I_{\hat{\cS}_1,u_2}(T)\right) 
      - \eps.
    \end{equation}
    Since, by \Cref{def:ValuesGame},
    \begin{equation*}
      V\left(S_{\hat{\cS}_1,u_2}(T), I_{\hat{\cS}_1,u_2}(T)\right) 
      = \adjustlimits\inf_{\mathcal{S}_1 \in \Gamma} 
      \sup_{\tilde u_2 \in \cU_2}
      \cF\left(S_{\hat{\cS}_1,u_2}(T), I_{\hat{\cS}_1,u_2}(T),
        \mathcal{S}_1(\tilde u_2), \tilde u_2\right),
    \end{equation*}
    then there exists $\check{\mathcal{S}}_1 \in \Gamma$ such that
    \begin{equation}
      \label{eq:V-new}
      V\left(S_{\hat{\cS}_1,u_2}(T), I_{\hat{\cS}_1,u_2}(T)\right)  + \eps \ge
      \cF\left(S_{\hat{\cS}_1,u_2}(T), I_{\hat{\cS}_1,u_2}(T), \check{\mathcal{S}}_1(u_2(\cdot + T)), u_2
      (\cdot + T)\right).
    \end{equation}
    Therefore, by~\eqref{eq:CostFunctional}, \eqref{eq:11},
    and~\eqref{eq:V-new},  we get
    \begin{align}
      \nonumber
      \tilde V\left(S_o, I_o\right)
       & \ge       \int_0^T \e^{-\theta t}
      \ell\left(I_{\hat{\cS}_1,u_2}(t),\hat{\cS}_1(u_2)(t), u_2(t)\right) \dd{t}
      \\
      \nonumber
      & \quad
        + \e^{-\theta T}
        \cF\left(S_{\hat{\cS}_1,u_2}(T), I_{\hat{\cS}_1,u_2}(T), \check{\mathcal{S}}_1 (u_2(\cdot + T)), u_2  (\cdot + T)\right) 
        - \eps \left(1 + \e^{\theta T}\right)
      \\
      \label{eq:estimate-tilde-V}
      & = \int_0^{+\infty} \e^{-\theta t}
        \ell\left(I_{\overline{\cS}_1,u_2}(t),\overline{\cS}_1(u_2)(t), u_2(t)\right) \dd{t}
        - \eps \left(1 + \e^{\theta T}\right),
    \end{align}
    where we use the non-anticipating strategy $\overline{\cS}_1\in \Gamma$,
    defined as
    \begin{equation}
      \label{eq:strategy-S1-bar}
      \begin{array}{rccl}
        \overline{\cS}_1:
        & \cU_2
        & \longrightarrow
        & \cU_2
        \\
        & u_2
        & \longmapsto
        & \overline{\cS}_1(u_2)(t) \coloneqq
          \left\{
          \begin{array}{ll}
            \hat{\cS}(u_2)(t),
            & \textrm{ if } t \in [0, T],
            \\
            \check{\cS}(u_2(\cdot + T))(t - T),
            & \textrm{ if } t > T.
          \end{array}
              \right.
      \end{array}
    \end{equation}
    Since $u_2$ is arbitrary, by~\eqref{eq:estimate-tilde-V} we deduce that
    \begin{align*}
      \nonumber
      \tilde V\left(S_o, I_o\right)
      & \ge \sup_{u_2 \in \cU_2}\left\{\int_0^{+\infty} \e^{-\theta t}
        \ell\left(I_{\overline{\cS}_1,u_2}(t),\overline{\cS}_1(u_2)(t), u_2(t)
        \right) \dd{t}
        \right\}
        - \eps \left(1 + \e^{\theta T}\right)
      \\
      & \ge V(S_o, I_o) -  \eps(1 + \e^{-\theta T}),
    \end{align*}
    proving that $\tilde V(S_o, I_o) \ge V(S_o, I_o)$ by the arbitrariness
    of $\eps>0$.
    \smallskip

    \textbf{We now prove that} $V(S_o, I_o) \geq \tilde{V}(S_o, I_o)$.

    Let $\eps > 0$ be fixed.
    Choose a strategy $\hat{\cS}_1 \in \Gamma$ such that 
    \begin{equation}
      \label{eq:proofDPP2a}
      V(S_o, I_o) \geq \sup_{u_2 \in \cU_2} 
      \int_0^{+\infty} \e^{-\theta t}
      \ell\left(I_{\hat{\cS}_1,u_2}(t), \hat{\cS}_1(u_2)(t), u_2(t)\right) \dd{t}
      - \eps,
    \end{equation}
    and then a control $\hat{u}_2 \in \cU_2$ such that 
    \begin{equation}
      \label{eq:proofDPP2b}
      \tilde{V}(S_o, I_o) \leq \int_0^{T} \e^{-\theta t} 
      \ell\left(I_{\hat{\cS}_1,\hat u_2}(t), \hat{\cS}_1(\hat{u}_2)(t),
        \hat{u}_2(t)\right) \dd{t} 
      + \e^{-\theta T} V\left(S_{\hat{\cS}_1,\hat u_2}(T),
        I_{\hat{\cS}_1,\hat u_2}(T)\right) + \eps,
    \end{equation}
    where the couple $(S_{\hat{\cS}_1,\hat u_2},I_{\hat{\cS}_1,\hat u_2})$
    denotes the solution to~\eqref{eq:VaxSystem} with initial datum
    $(S_o,I_o)$ and controls $\hat{\cS}_1(u_2),\hat u_2$.
    Define the map $\mathcal C: \cU_2 \to \cU_2$ such that,
    for all $u_2 \in \cU_2$,
    \begin{equation}
      \label{eq:12}
        \mathcal{C}(u_2)(t) \coloneqq 
        \left\{ 
        \begin{array}{ll}
            \hat{u}_2(t), & \textrm{ if } \; t \leq T, \\
            u_2(t-T), & \textrm{ if } \; t > T,
        \end{array}
        \right. 
    \end{equation}
    and define the non-anticipating strategy
    $\overline{S}_1 \in \Gamma$ by
    \begin{equation}
      \label{eq:proofDPP2c}
      \overline{\cS}_1(u_2)(t) \coloneqq 
      \hat{\cS}_1(\mathcal{C}({u}_2))(t + T),
    \end{equation}
    where $u_2 \in \cU_2$ and $ t > 0$.
    Moreover, fix a control $\check{u}_2 \in \cU_2$ such that
    \begin{equation}
    \label{eq:proofDPP2d}
    V\left(S_{\hat{\cS}_1,\hat u_2}(T),I_{\hat{\cS}_1,\hat u_2}(T)\right)
    \leq \int_0^{+\infty} \e^{-\theta t} 
    \ell\left(\tilde I(t), \overline{\cS}_1
    (\check{u}_2)(t), \check{u}_2(t)\right) \dd{t} + \eps,
  \end{equation}
  where $\left(\tilde S, \tilde I\right)$ is the solution
  to~\eqref{eq:VaxSystem} with initial datum
  $\left(S_{\hat{\cS}_1,\hat u_2}(T),I_{\hat{\cS}_1,\hat u_2}(T)\right)$
  and controls $\overline{\mathcal S}_1(\check u_2)$ and $\check u_2$.
  By construction, we easily deduce that
  \begin{equation}
    \label{eq:proofDPP2e} 
      \int_0^{+\infty} \e^{-\theta t}
      \ell\left(\tilde I(t), \overline{\cS}_1(\check{u}_2)(t), \check{u}_2(t)\right) \dd{t}
      =
       \int_T^{+\infty} \e^{-\theta (t-T)}
      \ell\left(I_{\hat{\cS}_1,\mathcal{C}(\check{u}_2)}(t), \hat{\cS}_1(\mathcal{C}(\check{u}_2))(t),
      \mathcal{C}(\check{u}_2)(t)\right) \dd{t}.
  \end{equation}
  Therefore, using~\eqref{eq:proofDPP2a}, \eqref{eq:proofDPP2b},
  \eqref{eq:12}, \eqref{eq:proofDPP2c}, \eqref{eq:proofDPP2d},
  and~\eqref{eq:proofDPP2e} we deduce that
  \begin{align*}
    \tilde{V}(S_o, I_o) 
    & \leq \int_0^{T} \e^{-\theta t} 
      \ell\left(I_{\hat{\cS}_1,\hat u_2}(t), \hat{\cS}_1(\hat{u}_2)(t), \hat{u}_2(t)\right) \dd{t} 
      + \e^{-\theta T} V\left(S_{\hat{\cS}_1,\hat u_2}(T),I_{\hat{\cS}_1,\hat u_2}(T)\right)
      + \eps
    \\
    & \leq \int_0^{T} \e^{-\theta t} 
      \ell\left(I_{\hat{\cS}_1,\hat u_2}(t), \hat{\cS}_1(\hat{u}_2)(t), \hat{u}_2(t)\right) \dd{t}
    \\
    & \quad + \e^{-\theta T} \int_0^{+\infty} \e^{-\theta t} 
      \ell\left(\tilde I(t), \overline{\cS}_1(\check{u}_2)(t), \check{u}_2(t)\right) \dd{t}
      + \eps\left(1 +\e^{-\theta T}\right)
    \\
    & \leq \int_0^{T} \e^{-\theta t} 
      \ell\left(I_{\hat{\cS}_1,\mathcal{C}(\check{u}_2)}(t),
      \hat{\cS}_1(\mathcal C(\check{u}_2))(t),
      \mathcal C(\check{u}_2)(t)\right) \dd{t}
    \\
    & \quad + \e^{-\theta T} \int_T^{+\infty} \e^{-\theta (t-T)} 
      \ell\left(I_{\hat{\cS}_1, \mathcal{C}(\check{u}_2)}(t), \hat{\cS}_1(\mathcal C(\check{u}_2))(t),
      \mathcal C(\check{u}_2)(t)\right) \dd{t}
      + \eps\left(1 +\e^{-\theta T}\right)
    \\ 
    & = \int_0^{+\infty} \e^{-\theta t} 
     { \ell\left(I_{\hat{\cS}_1,\mathcal{C}(\check{u}_2)}(t), \hat{\cS}_1(\mathcal C(\check{u}_2))(t),
      \mathcal C(\check{u}_2)(t)\right) \dd{t} }
      + \eps\left(1 +\e^{-\theta T}\right)
    \\
    & \leq \sup_{u_2 \in \cU_2}
      \left\{\int_0^{+\infty} \e^{-\theta t} 
      \ell\left(I_{\hat{\cS}_1, u_2}(t), \hat{\cS}_1({u}_2)(t), {u}_2(t)\right) \dd{t} \right\} 
      + \eps\left(1 +\e^{-\theta T}\right)
    \\
    & \leq V(S_o, I_o) + \eps \left(2 + \e^{-\theta T}\right)\; .  
  \end{align*}
    Take the limit as $\eps \to 0^+$ to the obtain the second inequality, which concludes the proof.
\end{proof}

\subsubsection{Hamilton-Jacobi equations}

First we deduce the regularity of the Hamiltonian functions.

\begin{lemma}
\label{lmm:ContinuityHamiltonian1}
    Set $\E \coloneqq \HH{1}((0, 1); \R) \times \R \times \LL2((0,1);\R) \times \R$. 
    Assume~\ref{hyp:(alpha)}-\ref{hyp:(F)}-\ref{hyp:(G)}. Let $\underline{\cH}$ and
    $\overline{\cH}$ be defined by~\eqref{eq:Hamiltonian1}
    and~\eqref{eq:Hamiltonian2}. 
    Then $\underline{\cH}, \overline{\cH} \in \Czero(\E; \R)$.
\end{lemma}

\begin{proof}
    It suffices to check that for all $u_1 \in [0, M_1]$ and $u_2 \in [0, M_2]$, the 
    pre-Hamiltonian $\cH$, defined by \eqref{eq:pre-Ham}, is continuous with respect to its 
    first two variables. To show it, we prove that $\cH$ is Lipschitz continuous on bounded 
    subsets of $\E$. Let $\Omega$ be a bounded subset of $\E$ and let $c \in \R^+$ such that
    \[
        \forall (S, I, p, q) \in \Omega, \quad 
        \max\{\|S\|_{\HH{1}}, |I|, \|p\|_{\LL2}, |q|\} \leq c.
    \]
    Let $u_1 \in [0, M_1]$ and $u_2 \in [0, M_2]$. For all $(\bY_1, \bP_1), (\bY_2, \bP_2) \in \Omega$, we have 
    \begin{align*}
        & |\cH(\bY_2, \bP_2, u_1, u_2) - \cH(\bY_1, \bP_1, u_1, u_2)|
        \\
        \leq &
        |\langle\p_\xi (gS_2) + (f + \alpha(I_2)) S_2 , p_2 - p_1 \rangle_{\LL{2}}|
        + |\langle\p_\xi (gS_2 - gS_1) + f(S_2 - S_1), p_1\rangle_{\LL{2}}| 
        \\
        & + |\langle \alpha(I_2) (S_2 - S_1), p_1\rangle_{\LL{2}}| 
        + |\langle (\alpha(I_2) - \alpha(I_1)) S_1, p_1\rangle_{\LL{2}}| 
        \\
        & + \beta |I_2 - I_1| \cdot |q_1| + \beta |I_1| \cdot |q_2 - q_1| 
        +  \left| \alpha(I_2) q_2 \int_0^1 (S_2 (\xi) - S_1(\xi)) \dd \xi \right|
        \\
        & + \alpha(I_2) |q_2 - q_1| \cdot \left| \int_0^1 S_1 (\xi) \dd \xi \right|
        + \left|\alpha(I_2) -\alpha(I_1) \right| \cdot \left| q_1 \int_0^1 S_1 (\xi) \dd \xi \right| 
        + |\ell(I_2, u_1,u_2) - \ell(I_1, u_1,u_2)| 
        \\
        \leq & \left( \gamma + \sup_{\xi \in [0, 1]} |g(\xi, u_1, u_2)| + 1 + 
        \sup_{I \leq c} \alpha(I) \right) \|S_2\|_{\HH{1}} \|p_2 - p_1\|_{\LL{2}}
        \\
        & + \left( \gamma + \sup_{\xi \in [0, 1]} |g(\xi, u_1, u_2)| + 1 \right) \|S_2 - S_1\|_{\HH{1}} \|p_1\|_{\LL{2}} 
        + \sup_{I \leq c} \alpha(I) \|S_2 - S_1\|_{\LL{2}} \|p_1\|_{\LL{2}} 
        \\
        & + \Lip(\alpha) |I_2 - I_1| \cdot \|S_1\|_{\LL{2}} \|p_1\|_{\LL{2}} 
        + \beta |I_2 - I_1| \cdot |q_1| + \beta |I_1| \cdot |q_2 - q_1| 
        + \sup_{I \leq c} \alpha(I) |q_2| \cdot \|S_2 - S_1\|_{\LL{2}} 
        \\
        & + \sup_{I \leq c} \alpha(I) |q_2 - q_1| \cdot \|S_1\|_{\LL{2}} 
        + \Lip(\alpha) |I_2 - I_1| \cdot |q_1| \cdot \|S_1\|_{\LL{2}} + 
        \kappa |I_2 - I_1| 
        \\
        \leq 
        & C \max\{\|S_2 - S_1\|_{\HH{1}}, |I_2 - I_1|, \|p_2 - p_1 \|_{\LL2}, 
        |q_2 - q_1|\},
    \end{align*}
    concluding the proof.
\end{proof}

We now tackle the proof of~\Cref{th:ValueViscosity1}, which is a direct 
adaptation of~\cite[Theorem 1.10, Chapter VIII]{BCD1997} to our infinite dimensional setting. The following two results are of use below.

\begin{lemma}
    \label{lmm:ChainRule}
    Assume~\ref{hyp:(alpha)}-\ref{hyp:(F)}-\ref{hyp:(G)}. 
    Let $u_1 \in \cU_1$, $u_2 \in \cU_2$ and 
    $\mathbf{Y}_o = (S_o, I_o) \in \HH{1}((0, 1); \R) \times \R$. Denote by 
    $\mathbf{Y} = (S, I)$ be the solution of~\eqref{eq:VaxSystem} with initial datum 
    $\mathbf{Y}_o$ and controls $u_1$, $u_2$. Then, for all 
    $\phi \in \Ck{1}(\LL{2}((0, 1); \R) \times \R; \R)$
    and for all $t > 0$,
    \begin{align*}
        \frac{\dd{}}{\dd{t}} \phi(\mathbf{Y}(t)) 
        & = -\scal{\nabla_1 \phi (\mathbf{Y}(t))}{\p_\xi (gS(t)) + (f(\xi) + \alpha(I(t)))S(t)}_{\LL{2}} 
        \\
        & \quad + \nabla_2 \phi (\mathbf{Y}(t)) 
        \left( -\beta I(t) + \alpha(I(t)) \int_0^1 S(t, \xi)\dd\xi \right),
    \end{align*}
    where we denoted by $\nabla_1 \phi$  and $\nabla_2 \phi$
    the gradient w.r.t.~the first variable of $\phi$
    and, respectively, w.r.t.~the second variable of $\phi$.
\end{lemma}
The proof is standard and so we omit it.

\begin{lemma}{\cite[Lemma~1.11, Chapter~VIII]{BCD1997}.}
    \label{lmm:ValueViscosity1}
    Let $\underline{\cH}$ be defined by~\eqref{eq:Hamiltonian1}. Let 
    $\mathbf{Y}_o \in \HH{1}((0, 1); \R) \times \R$ and 
    $\phi \in \Ck{1}(\LL{2}((0, 1); \R) \times \R)$ be such that 
    \[
        \nu \coloneqq \theta \phi(\mathbf{Y}_o) + \underline{\cH}(\mathbf{Y}_o, \nabla \phi (\mathbf{Y}_o)) > 0.
    \]
    Then, there exists $\cS_1^* \in \Gamma$ and $\tau > 0$ such that for all 
    $u_2 \in \cU_2$ and for all $t \in (0, \tau)$, 
    \[
        \int_0^t \ell(I(s), \cS_1^*(u_2), u_2) \e^{-\theta s} \dd{s} 
         + \phi(\mathbf{Y}(t)) \e^{-\theta t} - \phi(\mathbf{Y}_o)\leq -\frac{\nu t}{4},
    \]
    where $\mathbf{Y} = (S, I)$ is the solution of~\eqref{eq:VaxSystem} with the initial datum $\mathbf{Y}_o$ and controls $(\cS_1^*(u_2), u_2))$.
\end{lemma}
The proof is exactly the same of that of \cite[Lemma~1.11, Chapter~VIII]{BCD1997} even if the domain here is infinite dimensional.

\begin{proof}[\textbf{Proof of~\Cref{th:ValueViscosity1}}]
    We only prove the first item of~\Cref{th:ValueViscosity1},
    the second one being completely similar.
    
    \textbf{Regularity.} $V \in \Czero(\HH{1}((0, 1); \R) \times \R; \R)$. 
    It suffices to prove that for all $\cS_1 \in \Gamma$ and for all $u_2 \in \cU_2$, 
    $\cF$ defined by~\eqref{eq:CostFunctional} is continuous with respect to its first 
    two arguments. 
    Fix $(\overline{S}_o, \overline{I}_o) \in \LL{2}((0,1); \R) \times \R^+$ and
    consider a sequence $(S_o^n, I_o^n)$ converging to 
    $(\overline{S}_o, \overline{I}_o)$ in $\LL{2}((0,1); \R) \times \R^+$.
    We claim that
    \begin{equation*}
        \lim_{n \to + \infty} \mathcal{F}(S_o^n, I_o^n, \mathcal{S}_1(u_2), u_2)
        = \mathcal{F}(\overline{S}_o, \overline{I}_o, \mathcal{S}_1(u_2), u_2).
    \end{equation*}
    Denote by $(S^n, I^n)$, resp. $(\overline{S}, \overline{I})$, the
    solution to~\eqref{eq:VaxSystem} with initial datum 
    $(S_o^n, I_o^n)$, resp. $(\overline{S}_o, \overline{I}_o)$, and controls $\cS_1(u_2), u_2$. 
    By~\eqref{eq:CostFunctionalBis1}, it is sufficient to prove that
    \begin{equation}
        \label{eq:convergence-V-cont}
        \lim_{n \to + \infty} \int_0^{+\infty} e^{-\theta t} I^n(t) \dd t
        = \int_0^{+\infty} e^{-\theta t} \overline{I}(t) \dd t.
    \end{equation}
    First note that~\Cref{th:Stability} implies that the sequence $I^n$ converges pointwise to $\overline{I}$.
    
    Fix now an arbitrary subsequence $I^{n_h}$.
    By~\Cref{th:WellPosedness} we deduce that there exists
    $\bar n \in \N$ such that, for every $n \ge \bar n$ and $t > 0$,
    \begin{equation*}
        0 \le I^n(t) \le \norm{S_o^n}_{\LL1((0,1); \R)} + I_o^n
        \le \norm{S_o^n}_{\LL2((0,1); \R)} + I_o^n
        \le \norm{\overline{S}_o}_{\LL2((0,1); \R)} + \overline{I}_o + 1.
    \end{equation*}
    Therefore we deduce the existence of an additional subsequence, say $I^{n_{h_k}}$,
    and a function $\tilde I$ such that $I^{n_{h_k}}$ converges to $\tilde I$
    in the weak$^*$ topology of $\LL\infty((0, +\infty); \R)$.
    Let $\varphi \in \LL1((0, +\infty); \R)$; then
    \begin{equation*}
        \int_0^{+\infty} \varphi(t) \tilde I(t) \dd t
        = \lim_{k \to +\infty} \int_0^{+\infty} \varphi(t) I^{n_{h_k}}(t) \dd t
        = \int_0^{+\infty} \varphi(t) \overline{I}(t) \dd t,
    \end{equation*}
    where the last equality is due to the Lebesgue Dominated Convergence Theorem.
    This implies, by the arbitrariness of $\varphi$, that $\overline{I}(t) = \tilde I(t)$
    for a.e.~$t > 0$.
    The Uryshon subsequence principle permits to deduce that the whole sequence
    $I^{n}(t)$ weakly$^*$ converges to
    $\overline{I}(t)$ in 
    $\LL\infty((0,+\infty);\R)$.
    Therefore~\eqref{eq:convergence-V-cont} holds, concluding the proof of the continuity of $V$.
    
    \textbf{Sub-solution.} Let $\phi \in \Ck{1}(\LL{2}((0, 1); \R) \times \R; \R)$ 
    and $\mathbf{Y}_o \in \HH{1}((0, 1); \R) \times \R$ a local maximum point of $V - \phi$. Without loss of generality, we can assume that $V(\mathbf{Y}_o) - \phi(\mathbf{Y}_o) = 0$. We want to show that 
    \[
        \theta V(\mathbf{Y}_o) + \underline{\cH}(\mathbf{Y}_o, \nabla \phi (\mathbf{Y}_o)) \leq 0.
    \]
    Suppose instead that $\nu \coloneqq \theta V(\mathbf{Y}_o) + \underline{\cH}(\mathbf{Y}_o, \nabla \phi (\mathbf{Y}_o)) > 0$. Use~\Cref{lmm:ValueViscosity1} to find $\cS_1^* \in \Gamma$ and $\tau > 0$ 
    such that for all controls $u_2 \in \cU_2$ and for all $t \in (0, \tau)$,
    \begin{equation}
        \label{eq:ProofSubSolution}
        \int_0^t \ell(I(s), \cS_1^*(u_2), u_2) \e^{-\theta s} \dd{s} 
         + \phi(\mathbf{Y}(t)) \e^{-\theta t} - \phi(\mathbf{Y}_o)
         \leq -\frac{\nu t}{4},
    \end{equation}
    where $\mathbf{Y} = (S, I)$ is the solution of \eqref{eq:VaxSystem} with initial datum 
    $\mathbf{Y}_o$ and controls $(\cS_1^*(u_2), u_2)$.
    
    Notice that since $\mathbf{Y}_o$ is a local maximum of $V - \phi$, for all 
    $t \in ]0, \tau[$, we have 
    \[
        V(\mathbf{Y}(t)) - \phi(\mathbf{Y}(t)) \leq 0 
        \implies V(\mathbf{Y}(t)) \e^{-\theta t} \leq \phi(\mathbf{Y}(t)) \e^{-\theta t},
    \]
    which inserted in~\eqref{eq:ProofSubSolution} leads to
    \[
        \int_0^t \ell(I(s), \cS_1^*(u_2), u_2) \e^{-\theta s} \dd{s} 
         + V(\mathbf{Y}(t)) \e^{-\theta t} - V(\mathbf{Y}_o) \leq -\frac{\nu t}{4}.
    \]
    This contradicts \eqref{eq:DPP}, therefore, $V$ is a viscosity sub-solution.
    \smallskip 

    \textbf{Super-solution.} Let $\phi \in \Ck{1}(\LL{2}((0, 1); \R) \times \R; \R)$ 
    and $\mathbf{Y}_o = (S_o, I_o) \in \HH{1}((0, 1); \R) \times \R$ a local minimum point of $V - \phi$. Without loss of generality, we can assume that $V(\mathbf{Y}_o)-\phi(\mathbf{Y}_o)=0$. 
    We want to show that 
    \[
        \theta V(\mathbf{Y}_o) + \underline{\cH}(\mathbf{Y}_o, \nabla \phi (\mathbf{Y}_o)) \geq 0.
    \]
    Suppose instead that $\theta V(\mathbf{Y}_o) + \underline{\cH}(\mathbf{Y}_o, 
    \nabla \phi (\mathbf{Y}_o)) =: -\nu < 0$. 
    Consequently, there exists $u_2^* \in [0, M_2]$ such that for all 
    $u_1 \in [0, M_1]$,
    \[
    \begin{aligned}
        & \theta \phi (\mathbf{Y}_o) + 
        \left\langle \p_\xi (g(\cdot, u_1, u_2^*)S_o) + (f(\xi) + \alpha(I_o)) S_o , \nabla_1 \phi (\mathbf{Y}_o) \right\rangle_{\LL{2}} \\
        & + \left(\beta I_o - \alpha(I_o) \int_0^1 S_o \right) \nabla_2 \phi (\mathbf{Y}_o) 
    - \ell(I_o, u_1, u_2^*) \leq -\nu.
    \end{aligned}
    \]
    For all strategies $\cS_1 \in \Gamma$, for all $t > 0$ small enough, we have 
     \[
     \begin{aligned}
         & \theta \phi (\mathbf{Y}(t)) + 
    \left\langle\p_\xi (g(\cdot, \cS_1(u_2^*), u_2^*)S(t)) + (f(\xi) + \alpha(I(t)))S(t), 
    \nabla_1 \phi (\mathbf{Y}(t)) \right\rangle_{\LL{2}} \\
    & + \left(\beta I_o - \alpha(I(t)) \int_0^1 S(t) \right) \nabla_2 \phi (\mathbf{Y}(t)) 
    - \ell(I(t)), \cS_1(u_2^*), u_2^*) \leq -\nu,
     \end{aligned}
    \]
    where $\mathbf{Y} = (S, I)$ is the solution to~\eqref{eq:VaxSystem} with initial datum 
    $\mathbf{Y}_o$ and controls $(\cS_1^*(u_2), u_2)$. In view of~\Cref{lmm:ChainRule}, 
    the previous inequality can be re-written, for all $t > 0$ small enough, as 
    \[
        \theta \phi (\mathbf{Y}(t)) -\frac{\dd{}}{\dd{t}} \phi(\mathbf{Y}(t)) 
        - \ell(I(t), \cS_1(u_2^*), u_2^*) \leq -\nu.
    \]
    Multiply by $\e^{-\theta t}$ and integrate:
    \begin{equation}
        \label{eq:ProofSuperSolution}
        \phi(\mathbf{Y}(t)) \e^{-\theta t} - \phi(\mathbf{Y}_o)
        + \int_0^t \ell(I(s), \cS_1(u_2^*), u_2^*) \e^{-\theta s} \dd{s} 
        \geq \frac{\nu}{\theta}(1 - \e^{-\theta t}).
    \end{equation}
    Notice that since $\mathbf{Y}_o$ is a local minimum of $V - \phi$, we have for all 
    $t > 0$ small enough,
    \[
        V(\mathbf{Y}(t)) - \phi(\mathbf{Y}(t)) \geq 0 
        \implies V(\mathbf{Y}(t)) \e^{-\theta t} \geq \phi(\mathbf{Y}(t)) \e^{-\theta t}.
    \]
    We insert it in~\eqref{eq:ProofSuperSolution} to obtain, for all $t > 0$ small enough,
    \[
        V(\mathbf{Y}(t)) \e^{-\theta t}
        + \int_0^t \ell(I(s), \cS_1(u_2^*), u_2^*) \e^{-\theta s} \dd{s} 
        \geq V(\mathbf{Y}_o) + \frac{\nu}{\theta}(1 - \e^{-\theta t}) > V(\mathbf{Y}_o),
    \]
    which contradicts~\eqref{eq:DPP}. This shows that $V$ is a viscosity super-solution.
\end{proof}

\section*{Acknowledgments}
MG and ER were partially supported by the 2024 INdAM GNAMPA project \textsl{Modelling and Analysis through Conservation Laws}.
The authors acknowledge the PRIN 2022 project \emph{Modeling, Control and Games through Partial Differential Equations} {D53D23005620006}, funded by the European Union-Next Generation EU.

MG was partially supported by the Project  funded  under  the
National  Recovery  and  Resilience  Plan  (NRRP)
of  Italian  Ministry  of  University and Research
funded by the European Union (NextGenerationEU Award, No.  CN000023,  Concession  Decree  No.  1033  of  17/06/2022  adopted  by  the  Italian Ministry of University and Research, CUP: H43C22000510001, Centro Nazionale per la Mobilità Sostenibile).

MG was also partially supported by the Project  funded  under  the
National  Recovery  and  Resilience  Plan  (NRRP)
of  Italian  Ministry  of  University and Research
funded by the European Union (NextGenerationEU Award, No.  
ECS\_00000037, CUP H43C22000510001,
MUSA -- Multilayered Urban Sustainability Action).

{\small
  \bibliography{VaxL2}
  \bibliographystyle{abbrv}
}

\end{document}